\documentclass[11pt,reqno,letterpaper]{amsart}
\usepackage{amssymb}
\usepackage{graphicx,color}
\usepackage{hyperref}
\usepackage[normalem]{ulem}

\usepackage[utf8]{inputenc}
\usepackage{gensymb}
\usepackage[square,numbers]{natbib}
\usepackage{amssymb}
\usepackage{amsfonts}
\usepackage{latexsym}
\usepackage{amsmath,systeme}
\usepackage{graphics}
\usepackage{subfigure}
\usepackage{graphicx}
\usepackage{epsfig}
\usepackage[leqno]{mathtools}
\usepackage{amsthm}
\usepackage{MnSymbol}
\usepackage[titletoc]{appendix}
\usepackage[dvipsnames,usenames]{xcolor}

\usepackage{fullpage} 

\mathtoolsset{showonlyrefs}

\usepackage{scalerel}[2014/03/10]
\usepackage[usestackEOL]{stackengine}

\def\dashint{\,\ThisStyle{\ensurestackMath{%
  \stackinset{c}{.2\LMpt}{c}{.5\LMpt}{\SavedStyle-}{\SavedStyle\phantom{\int}}}%
  \setbox0=\hbox{$\SavedStyle\int\,$}\kern-\wd0}\int}
\def\ddashint{\,\ThisStyle{\ensurestackMath{%
  \stackinset{c}{.2\LMpt}{c}{.5\LMpt+.2\LMex}{\SavedStyle-}{%
    \stackinset{c}{.2\LMpt}{c}{.5\LMpt-.2\LMex}{\SavedStyle-}{%
      \SavedStyle\phantom{\int}}}}\setbox0=\hbox{$\SavedStyle\int\,$}\kern-\wd0}\int}

\theoremstyle{plain}

\theoremstyle{plain}
\newtheorem{theorem}{Theorem} [section]

\newtheorem{lemma}[theorem]{Lemma}

\theoremstyle{definition}
\newtheorem{definition}[theorem]{Definition}
\theoremstyle{remark}
\newtheorem{remark}[theorem]{Remark}

\def\RR{{\mathbb R}}

\def\N{{\mathbb N}}
\def\LE{\mathcal{E}}

\numberwithin{theorem}{section}
\numberwithin{equation}{section}

\begin{document}

\title{Local and nonlocal energy-based coupling models}

\thanks{}

\author[G. Acosta]{Gabriel Acosta}

\author[F. Bersetche]{Francisco M. Bersetche}

\author[J. D. Rossi]{Julio D. Rossi}

\address{Departamento  de Matem\'atica, FCEyN, Universidad de Buenos Aires,
 Pabell\'on I, Ciudad Universitaria (1428),
Buenos Aires, Argentina \newline
\indent G. Acosta is also a member of IMAS, Conicet, Argentina.}
\email{gacosta@dm.uba.ar, fbersetche@dm.uba.ar, jrossi@dm.uba.ar}

\thanks{G.A. partially supported by ANPCyT under grant PICT 2018 - 3017 and CONICET grant PIP 11220130100184CO (Argentina). \\
\indent F.M.B. partially supported by PEDECIBA postdoctoral fellowship (Uruguay), and by ANPCyT under grant PICT 2018 - 3017 (Argentina).
 \\
\indent J.D.R. partially supported by 
CONICET grant PIP GI No 11220150100036CO
(Argentina), PICT 2018 - 3183 (Argentina) and UBACyT grant 20020160100155BA (Argentina).}

\date{}

\begin{abstract}
In this paper we study two different ways of coupling a local operator with a nonlocal one in such a way that the 
resulting equation is related to an energy functional.
In the first strategy the coupling is given via source terms in the equation and in the 
second one a flux condition in the local part appears. 
For both models we prove existence and uniqueness of a solution that is obtained
via direct minimization of the related energy functional.
In the second part of this paper we extend these ideas to deal with 
local/nonlocal elasticity models in which we couple classical local elasticity with nonlocal
peridynamics. 
\end{abstract}

\maketitle


\section{Introduction} 

Our main goal in this paper is to show existence and uniqueness of solutions 
to two different coupled local/nonlocal equations that are naturally associated with 
two different energies. We deal both with the scalar and the vectorial case (covering
elasticity models).

Nonlocal models can describe phenomena
not well represented by classical Partial Differential Equations, PDE, (including
problems characterized by long-range interactions and discontinuities). 
For instance, in the context of diffusion, long-range interactions effectively
describe anomalous diffusion, while in the context of mechanics, cracks
formation results in material discontinuities. 
The fundamental difference between nonlocal models and classical local models  
is the fact that the latter only involve differential operators, whereas the
former rely on integral operators. For general references on nonlocal models 
with applications to elasticity, population dynamics, image processing, etc, the list is quite large. For a glimpse we refer
to \cite{BCh,Bere,CaRo,CF,ChChRo,CERW,Cortazar-Elgueta-Rossi-Wolanski,Coville,DiPaola,F,Hutson,MenDu,Sil,Sil1,
Sil2,Strick,W,Z} and the book \cite{ElLibro}.

It is often the case that nonlocal effects are concentrated only in
some parts of the domain, whereas, in the remaining parts, the system can be accurately
described by a PDE. The goal of coupling local and nonlocal models is to
combine a local equation (a PDE) with a nonlocal one (an integral equation),
under the assumption that the location of local and nonlocal effects can be identified
in advance. 
In this
context, one of the challenges of a coupling strategy is to provide a mathematically 
consistent formulation.

From a mathematical point of view, interesting properties arise from coupling local and nonlocal models, 
see \cite{Peri2,Peri3,delia2,delia3,Du,Gal,GQR,Kri,santos2020} and references therein.
As previous examples of coupling approaches between 
local and nonlocal regions we refer the reader to \cite{Peri1,Peri2,Peri3,delia2,delia3,delia,Du,Gal,GQR,Han,Kri,santos2020,Sel,Sel2,Sel3}
the survey \cite{SUR}
and references therein. 
Previous strategies treat the coupling condition as an optimization
objective (the
goal is to minimize the mismatch of the local and nonlocal solutions on the overlap
of their sub-domains).
Another example relies on the partitioned procedure as a general coupling
strategy for heterogeneous systems, the system is divided into sub-problems in
their respective sub-domains, which communicate with each other via transmission
conditions. 
In \cite{Bere} the effects of network transportation on enhancing biological invasion is studied. The proposed mathematical model consists of one equation with nonlocal diffusion in a one-dimensional domain coupled via the boundary condition with a standard reaction-diffusion, in a two-dimensional domain. The results suggested that the fast diffusion enhances the spread in the domain in which the local diffusion takes place.
In \cite{delia2}, local and nonlocal problems were coupled through a prescribed region in which both kinds of equations overlap (the value of the solution in the nonlocal part of the domain is used as a Dirichlet boundary condition for the local part and vice-versa). This kind of coupling gives continuity of the solution in the overlapping region but does not preserve the total mass when Neumann boundary conditions are imposed.
In \cite{delia2} and \cite{Du}, numerical schemes using local and nonlocal equations were developed and used to improve the computational
accuracy when approximating a purely nonlocal problem. In \cite{GQR} and \cite{santos2020} (see also \cite{Gal, Kri}), 
evolution problems related to energies closely related to ours are studied (here we deal with stationary problems).

In this paper we introduce two different ways of coupling local and nonlocal models. 
Let us describe briefly what we have in mind and refer to the next section for the precise hypothesis
and statements of our results. 
Here we fix a bounded domain $\Omega \subset \mathbb{R}^N$ that is  
divided into two disjoint subdomains $\Omega_\ell,\Omega_{n\ell}\subset \Omega$, $\Omega_\ell \cup \Omega_{n\ell} = \Omega$,
$\Omega_\ell \cap \Omega_{n\ell}= \emptyset$.  
In the first one, $\Omega_\ell$, we have a local operator while in the second one, $\Omega_{n\ell}$,
we have a nonlocal operator. Therefore, we look for an energy that involves terms like
$$
\int_{\Omega_\ell} \frac{|\nabla u(x)|^2}{2} {\rm d}x 
$$
(this is related to a local operator, the Laplacian, in $\Omega_\ell$) and 
$$
\frac{1}{2}\int_{\Omega_{n\ell}}\int_{\Omega_{n\ell}} J(x-y)(u(y)-u(x))^2\, {\rm d}y{\rm d}x
$$
(that gives a nonlocal operator in $\Omega_{n\ell}$).
The main issue is then to couple the two regions.
We deal here with two different couplings that we will call 
\emph{volumetric} and \emph{mixed} couplings. Volumetric couplings describe interactions between sets of positive $N-$dimensional measure while in mixed couplings volumetric parts of $\Omega_{n\ell}$ and lower dimensional parts of $\Omega_{\ell}$ can interact with each other. 
In the first case we add to our energy a term like 
$$
\frac{1}{2}\int_{\Omega_{\ell}}\int_{\Omega_{n\ell}} J(x-y)(u(y)-u(x))^2\, {\rm d}y{\rm d}x
$$
(notice that here we are integrating in $\Omega_{\ell} \times \Omega_{n\ell}$) 
and in the latter case we fix a hypersurface on the boundary 
of $\Omega_\ell$ (that we call $\Gamma$) and we add 
$$
 \frac{1}{2}\int_{\Omega_{nl}}\int_{\Gamma}G(x,z)\left(u(x)-u(z)\right)^2  {\rm d}\sigma(z) {\rm d}x
$$
(remark that here we are integrating in $\Omega_{nl}\times \Gamma$ and that $\Gamma$ is 
a lower dimensional set). 

Here we develop the theory looking for minimizers of energies that combine the previous 
terms. We show that under very general conditions on the domains and the kernels (conditions
that mainly ensure some connectedness of the underlying topology behind the resulting problem) 
we have existence and uniqueness of a minimizer that verifies a set of equations
coupling local and nonlocal operators. 
We deal both with the scalar and the vectorial case (in this last case we are able to include couplings
between classical local elasticity models and the nonlocal elasticity model called peridynamics). 
To make the exposition as simple as possible we consider homogeneous
Dirichlet boundary/exterior data. 

\medskip

The paper is organized as follows: in \ref{sec:main} we introduce the precise conditions
that we impose on our sets and kernels and we state our main results; in \ref{sect-scalar} we 
deal with the scalar case and in \ref{sect-vectorial} we tackle the more involved vectorial case including
  elasticity models; finally, in \ref{sect-comments} we mention some possible extensions of our results.

\section{Main Results}
\label{sec:main}

Let us describe in detail the general setting that we consider here.
Along this work we consider an open bounded domain $\Omega \subset \mathbb{R}^N$. In our models it is 
assumed that $\Omega$ is  
divided into two disjoint subdomains $\Omega_\ell,\Omega_{n\ell}\subset \Omega$.  
The first one, $\Omega_\ell$, is where local phenomena take place and the second one, $\Omega_{n\ell}$,
is the domain where nonlocal effects occur. 
We assume that we can write
 $ \Omega = (\overline{\Omega}_\ell \cup \overline{\Omega}_{n\ell})^{\circ}$
(both subdomains ${\Omega}_{\ell}$ and ${\Omega}_{n\ell}$ are assumed to be open). Notice that it may happen that
$\partial{\Omega}_\ell \cap \partial \Omega \neq \emptyset$ or
$\partial{\Omega}_\ell \subset \Omega$ as illustrated in Figure 1
(it also may happen that $\partial{\Omega}_{n\ell} \cap \partial \Omega \neq \emptyset$ 
or $\partial{\Omega}_{n\ell} \subset \Omega$).

\begin{figure}[h]\label{Figura1}
\begin{center}
\includegraphics[scale=.4]{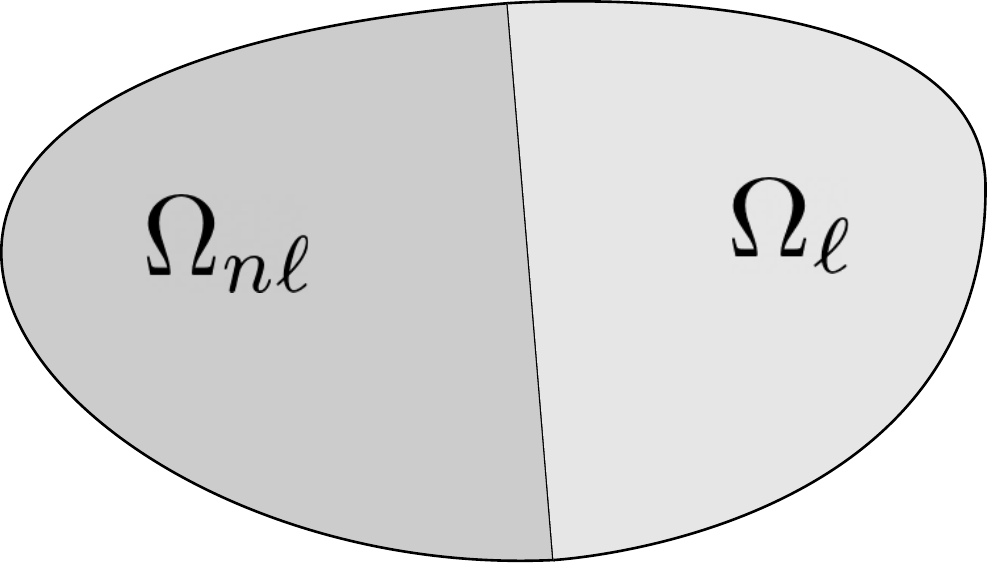} \qquad
	\includegraphics[scale=.4]{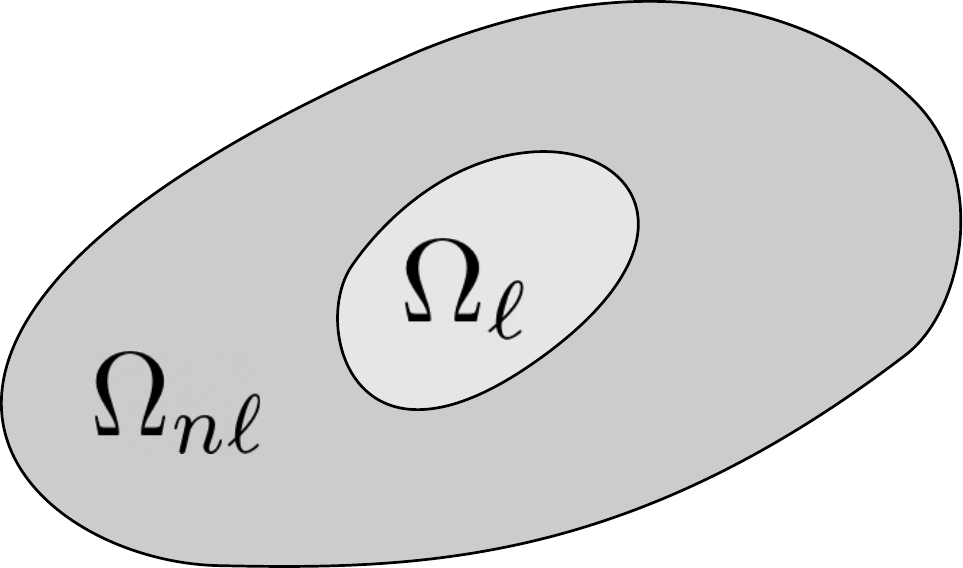}
		\caption{{Two possible partitions $\Omega_\ell \cup \Omega_{n\ell} = \Omega$.}}
\end{center}
\end{figure}

We will consider a nonlocal operator that involves a kernel $J:\RR^N \to \RR$ that is a nonnegative measurable function. 
For $J$ we assume 
 \begin{itemize}
  \item[$(J1)$] there exist $\delta>0$
and $C>0$ such that $J(z)>C$ for all $z$ such that $\|z\|\le 2\delta$.


\item[$(J2)$] the convolution $T_J(f)=J*f$ defines a compact operator in $L^2(\Omega_{n\ell})$.
 \end{itemize}
 
In our models, $J$ is a kernel that encodes the effect of a general \emph{volumetric} nonlocal interaction. Condition $(J1)$ guarantees the influence of nonlocality within an horizon of size at least $2\delta$ while $(J2)$ is a technical requirement fulfilled, for  instance, by continuous 
kernels, characteristic functions, or even for $L^2$ kernels, (this holds since these kernels produce Hilbert-Schmidt operators of the form 
$f \mapsto T(f)(x):=\int k(x,y) f(y) {\rm d}y$ that are compact if $k\in L^2$, see Chapter VI in \cite{brezis}). 

Now we need to introduce a connectivity condition. 

 \begin{definition}\label{def:deltaconnected}
 {\rm We say that an open set $D \subset \mathbb{R}^N$ is $\delta-$connected , with $\delta\ge 0$, if it can not be written as a disjoint union of two
 (relatively) open nontrivial sets 
 that are at distance greater or equal than $\delta.$}
\end{definition}

 Notice that if $D$ is $\delta$ connected then it is $\delta'$ connected for any $\delta'\ge \delta$. From \ref{def:deltaconnected}, 
we notice that $0-$connectedness agrees with the classical notion of being connected (in particular, open connected sets are $\delta-$connected).
\ref{def:deltaconnected} can be written in an equivalent way: an open set $D$ is $\delta-$connected if given 
  two points $x,y\in D$,  there exists a finite number of points $x_0,x_1,...,x_n \in D$ 
 such that $x_0=x$,
 $x_n=y$ and $dist(x_i,x_{i+1}) <\delta$. 
 
 Loosely speaking $\delta$-connectedness combined with $(J1)$ says that the effect of nonlocality can travel beyond the horizon $2\delta$ through the whole domain.

 With \ref{def:deltaconnected} at hand we can write the following extra assumptions on the local/nonlocal sets. 
 \begin{enumerate}
 \item[(1)] $ \Omega_\ell$ is connected and smooth ($ \Omega_\ell$ has Lipschitz boundary),
 
 
 \item[(2)] $ \Omega_{n\ell} $ is $\delta-$connected.
 \end{enumerate}
 
 In order to keep our presentation as simple as possible volumetric couplings are modeled by means of the same kernel $J$
(however, other choices are possible, see \ref{rem.1.5} below).

Mixed couplings, on the other hand, are  restricted to interactions of $\Omega_{n\ell}$ with a fixed smooth hypersurface  
$\Gamma \subset \partial \Omega_\ell.$ In this context a nonnegative and measurable function 
$G:\Gamma \times \Omega_{n\ell} \mapsto \mathbb{R},$
plays the role of the associated kernel. The following condition is a substitute of the volumetric counterpart.
 \begin{itemize}
  \item[$(G1)$]  there exist $\delta>0$
and $C>0$ such that for any $(x,y)\in \Gamma\times \Omega_{n\ell}$,  $G(x,y)>C$  if $\|x-y\|\le 2\delta$.
 \end{itemize}
Finally, in order to avoid trivial couplings in any of the two cases, we impose that  $\Omega_{\ell}$ and $\Omega_{n\ell}$ need to be closer than the horizon of the involved kernel, 
\begin{enumerate}
\item[$(P1)$]  $dist ( \Omega_{\ell},\Omega_{n\ell})<\delta$,
 
 
\item[$(P2)$]  $dist(\Gamma,\Omega_{n\ell})<\delta$.
\end{enumerate}

\begin{remark} \label{rem.intro.dom}
Our results are valid for more general domains. In fact, we assumed that $\Omega_\ell$ is connected and that
$\Omega_{n\ell}$ is $\delta-$connected with $dist(\Omega_{\ell},\Omega_{n\ell})<\delta$, but we can also handle
the case in which $\Omega_\ell$ has several connected components and $\Omega_{n\ell}$ has several 
$\delta-$connected components as long as they are close between them. See Remark \ref{re,.sec.dom} 
for extra details. We prefer to state our results under conditions (1), (2), $(G1)$ and $(P1)$, $(P2)$ just to simplify the presentation. 
\end{remark}

\subsection{Scalar problems}

 As a warm up we first deal with the scalar case, that is, we look for minimizers of local/nonlocal
energies with different coupling terms defined for real valued functions 
$
u: \Omega \mapsto \mathbb{R}.$
This case is simpler but it contains many of the difficulties that we face when dealing with the complete
elasticity local/nonlocal model. As it is usual when one deals with nonlocal problems one needs
$u$ to be defined outside $\Omega$ by imposing that $u$ agrees with a prescribed datum, see below.

Let us introduce the models and energies that we study.

\subsubsection{{\bf First model.} Coupling local/nonlocal problems via source terms}
Our aim is to look for a scalar problem with an energy that combines local and nonlocal terms
acting in different subdomains, $\Omega_\ell$ and $\Omega_{n\ell}$, of $\Omega$.

\begin{figure}[h]\label{fig:first_model}
\begin{center}
    	\includegraphics[scale=.4]{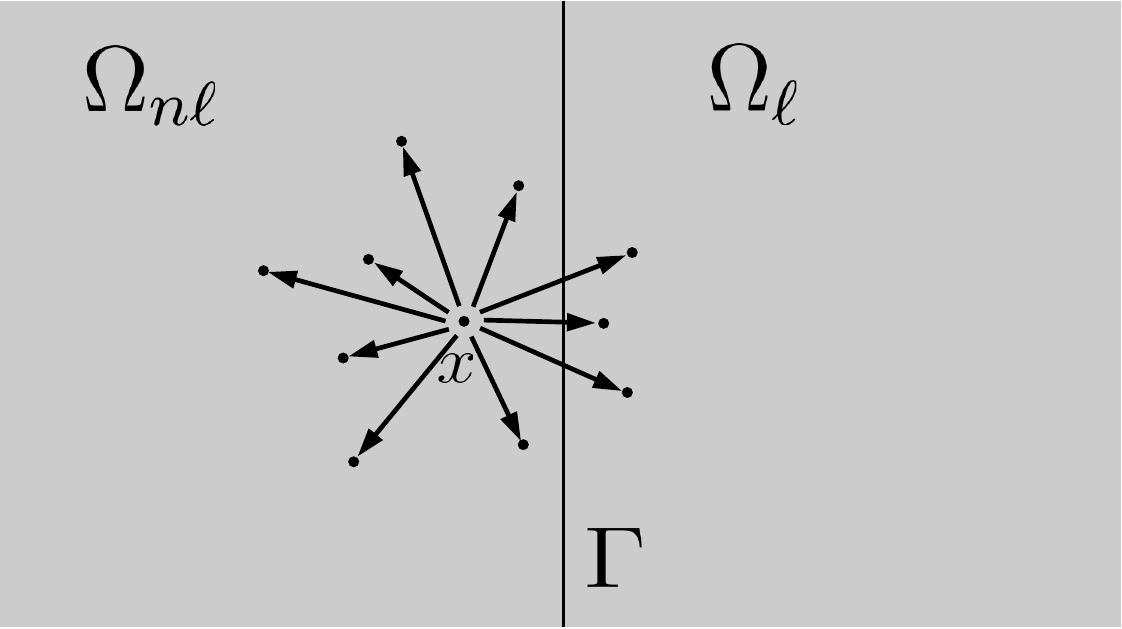}
		\caption{{Nonlocal interactions from $x$ in the first model}}
\end{center}
\end{figure}

We consider the following energy: 
\begin{equation}\label{def:funcional_dirichlet.intro.2299}
E_{i} (u):=\int_{\Omega_\ell} \frac{|\nabla u(x)|^2}{2} {\rm d}x + \frac{1}{2}\int_{\Omega_{n\ell}}\int_{\mathbb{R}^N} J(x-y)(u(y)-u(x))^2\, {\rm d}y{\rm d}x
- \int_\Omega f(x) u(x) {\rm d}x.
\end{equation}
Notice that here we are integrating in the whole $\mathbb{R}^N$, hence, to evaluate this energy we have to take a fixed exterior datum
(that we assume homogeneous for simplicity)
$
u(x) = 0$, for $ x \not\in \Omega.$
For a non-homogeneous datum $u=g_d\neq 0$ in $\mathbb{R}^N \setminus \Omega$ 
we just have to extend $g_d$ to the whole $\mathbb{R}^N$ and then consider the problem for
$w=u - g_d$. 

In this energy we have 
\begin{equation} \label{termino}
\frac{1}{2}\int_{\Omega_{n\ell}}\int_{\mathbb{R}^N} J(x-y)(u(y)-u(x))^2\, {\rm d}y{\rm d}x
\end{equation}
that can be decomposed as 
$$
\frac{1}{2}\int_{\Omega_{n\ell}}\int_{\mathbb{R}^N \setminus \Omega_\ell} J(x-y)(u(y)-u(x))^2\, {\rm d}y{\rm d}x + \frac{1}{2}\int_{\Omega_{n\ell}}\int_{\Omega_\ell} J(x-y)(u(y)-u(x))^2\, {\rm d}y{\rm d}x.
$$
The first part gives a nonlocal operator in $\Omega_{n\ell}$ while the second term encodes the coupling between 
$\Omega_{n\ell}$ and $\Omega_{\ell}$. In the double integral \eqref{termino} we observe that two points, $x$ and $y$,
interact if and only if one of them is in $\Omega_{n\ell}$.

Notice that the energy $E_{i}$ is well defined and finite for functions in the space
$$
H_{i} = \Big\{ u \in L^2 (\Omega), u|_{\Omega_\ell} \in H^1 (\Omega_\ell), 
  u = 0 \mbox{ in } \mathbb{R}^N \setminus \Omega \Big\}.
$$

Here we look for minimizers of $E_{i}$ in $H_{i}$. Our first result reads as follows:

\begin{theorem} \label{teo.1.22}
Assume $(J1)$, $(J2)$, $(1)$, $(2)$ and $(P1)$.
Given $f \in L^2 (\Omega)$ 
there exists a unique minimizer of  $E_{i}$ in $H_{i}$.
The unique minimizer is a weak solution to the equation 
\begin{equation}  \label{eq:main.Dirichlet.local.2277}
\begin{cases}
\displaystyle - f(x)=\Delta u (x) + \int_{\Omega_{n\ell}} J(x-y)(u(y)-u(x))\,{\rm d} y, &x\in \Omega_\ell,
\\[7pt]
\displaystyle \partial_\eta u(x)=0,\qquad & x\in  \partial \Omega_\ell \cap \Omega, \\[7pt]
u(x)= 0, & x \in \partial \Omega \cap \partial \Omega_\ell,
\end{cases}
\end{equation}
 in $\Omega_{\ell}$ and to the following  nonlocal equation in $\Omega_{n\ell}$, with the nonlocal Dirichlet exterior condition
\begin{equation}  \label{eq:main.Dirichlet.nonlocal.2277}
\left\{
\begin{array}{l} 
\displaystyle \!\! - f(x) = \!\! \int_{\mathbb{R}^N\setminus \Omega_{n\ell}} J(x-y)(\! u(y)-u(x)\! ) {\rm d}y
+ 2 \! \int_{\Omega_{n\ell}} J(x-y)  (\! u(y)-u(x) \!){\rm d}y,
\, x \in \Omega_{n\ell}, \\[7pt]
\displaystyle \!\! u(x)= 0,\qquad  x\in \mathbb{R}^N\setminus \Omega.
\end{array} \right.
\end{equation}
\end{theorem}

\begin{remark} \label{rem.1.5} {\rm 
We can consider more general energies such as 
\begin{equation}\label{def:funcional_dirichlet.intro.2277}
\begin{array}{l}
\displaystyle E_{i} (u):=\int_{\Omega_\ell} \frac{|\nabla u (x)|^2}{2} {\rm d}x 
+ \frac{1}{2}\int_{\Omega_{n\ell}}\int_{\mathbb{R}^N \setminus \Omega_\ell} J(x-y)(u(y)-u(x))^2\, {\rm d}y{\rm d}x \\[7pt]
\qquad\qquad \quad \displaystyle  + \frac{1}{2}\int_{\Omega_{n\ell}}\int_{\Omega_\ell} G(x-y)(u(y)-u(x))^2\, {\rm d}y{\rm d}x
- \int_\Omega f (x) u(x) {\rm d}x.
\end{array}
\end{equation}
Here there are two different kernels involved, $J$ and $G$. The use of a different kernel for the interactions
between $\Omega_\ell$ and $\Omega_{n\ell}$ can be handled with similar arguments (see below, where
in our second model we deal with two different kernels, the second one acting in sets of different dimension). 
In this first model, to simplify the presentation, we prefer to consider only one kernel and take $J=G$ in \eqref{def:funcional_dirichlet.intro.2277}.}
\end{remark}

\begin{remark} A probabilistic interpretation of this model in terms of particle systems runs as follows:
take an exponential clock that controls the jumps of the particles, 
in the local region $\Omega_{\ell}$ the particles move continuously according to Brownian motion 
(with a reflexion at $\Gamma$, the part of the boundary of $\Omega_\ell$ inside $\Omega$)
and when the clock rings a new position is sorted according to the kernel 
$J(x-\cdot)$, then they jump if the new position is in the nonlocal region $\Omega_{n\ell}$ 
(if the
sorted position lies outside $\Omega_{n\ell}$ then particles just continue moving by Brownian motion ignoring the
ring of the clock); while in the nonlocal region $\Omega_{n\ell}$ the particles stay still until the clock rings
and then they jump using the kernel $J(x-\cdot)$ to select the new position in the whole $\mathbb{R}^N$.
The particle gets killed when it arrives to $\mathbb{R}^N \setminus \Omega$ (coming from $\Omega_{\ell}$ by
Brownian motion or jumping from $\Omega_{n\ell}$).

The minimizer to our functional $E_{i}(u)$ gives the stationary distribution of particles provided
that there is an external source $f$ (that adds particles where $f>0$ and remove particles where $f<0$). 
\end{remark} 

\begin{remark} We can also consider the energy
\begin{equation}\label{def:funcional_dirichlet.intro.229944}
\widetilde{E}_{i} (u):=\int_{\Omega_\ell} \frac{|\nabla u(x)|^2}{2} {\rm d}x + \frac{1}{2}\int_{\mathbb{R}^N \setminus \Omega_{\ell}}
\int_{\mathbb{R}^N} J(x-y)(u(y)-u(x))^2\, {\rm d}y{\rm d}x
- \int_\Omega f (x) u(x) {\rm d}x.
\end{equation}
Notice that now the double integral takes place in the set $(\mathbb{R}^N \setminus \Omega_{\ell}) \times \mathbb{R}^N$.
Here we have that two points, $x$ and $y$,
interact if and only if one of them is not in $\Omega_{\ell}$ (therefore, now there are nonlocal interactions between $\Omega_{\ell}$
and $\mathbb{R}^N \setminus \Omega$ that were not present in the previous energy $E_i$).
.

With the same ideas used to deal with ${E}_{i}$ we can show that there is a unique minimizer of $\widetilde{E}_{i}$ in 
$
H_{i} = \{ u \in L^2 (\Omega), u|_{\Omega_\ell} \in H^1 (\Omega_\ell), 
 u = 0 \mbox{ in } \mathbb{R}^N \setminus \Omega \}.
$
In this case the limit problem reads as
\begin{equation}  \label{eq:main.Dirichlet.local.2277.878}
\begin{cases}
\displaystyle - f(x)=\Delta u (x) + \int_{\mathbb{R}^N \setminus \Omega_{\ell}} J(x-y)(u(y)-u(x))\,{\rm d} y, &x\in \Omega_\ell,\, t>0,
\\[7pt]
\displaystyle \partial_\eta u(x)=0,\qquad & x\in  \partial \Omega_\ell \cap \Omega, \\[7pt]
u(x)= 0, & x \in \partial \Omega \cap \partial \Omega_\ell,
\end{cases}
\end{equation}
and 
\begin{equation}  \label{eq:main.Dirichlet.nonlocal.2277.878}
\left\{
\begin{array}{l}
\displaystyle \!\!- f(x) = \!\! \int_{\mathbb{R}^N\setminus \Omega_{n\ell}} J(x-y)(\! u(y)-u(x) \!) {\rm d}y
+ 2 \! \int_{\Omega_{n\ell}} J(x-y)  (\! u(y)-u(x) \!){\rm d}y,
\, x \in \Omega_{n\ell}, \\[7pt]
\displaystyle u(x)= 0,\qquad  x\in \mathbb{R}^N\setminus \Omega.
\end{array} \right.
\end{equation}
\end{remark}

\subsubsection{{\bf Second model.} Coupling local/nonlocal problems via flux terms}
For the second model we fix a smooth hypersurface $\Gamma \subset \partial {\Omega_\ell} \cap \Omega,$
and use a different nonnegative kernel $G:\Gamma \times \Omega_{n\ell} \mapsto \mathbb{R}$
that will control the coupling between the local and the nonlocal regions.

\begin{figure}[h]\label{fig:second_model}
\begin{center}
    	\includegraphics[scale=.4]{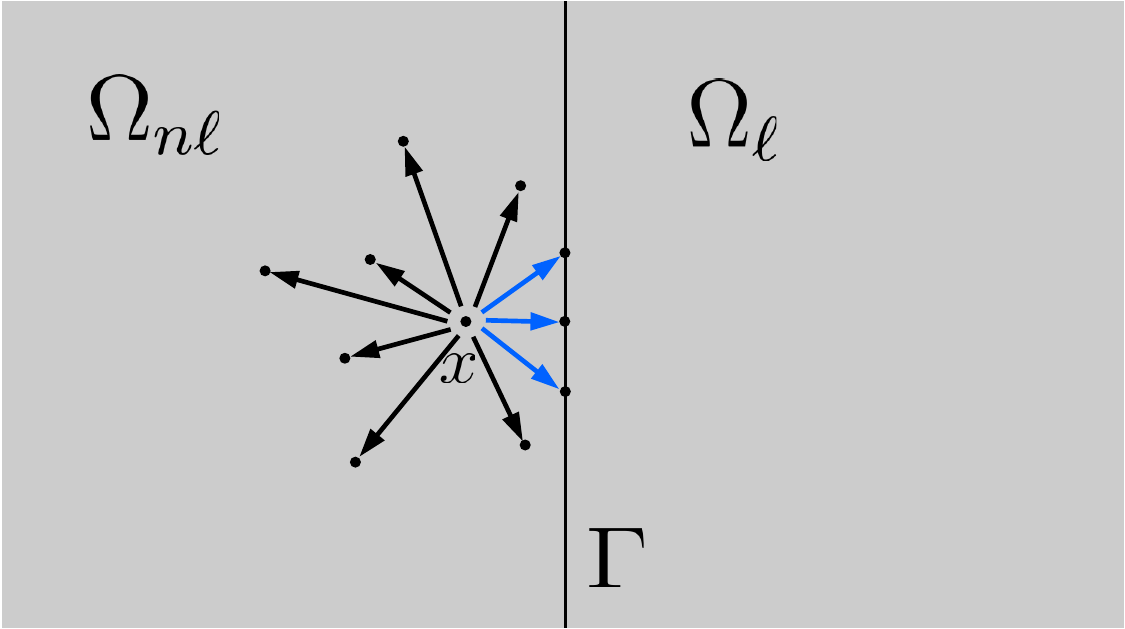}
	\caption{{Nonlocal interactions from $x$ in the second model}}
\end{center}
\end{figure}

Our aim now is to look for a scalar problem with an energy that combines local and nonlocal terms 
acting in different subdomains, $\Omega_\ell$ and $\Omega_{n\ell}$, of $\Omega$, but now the coupling
is made balancing the fluxes across the two subdomains, in the local domain we use the flux across $\Gamma$.

For a scalar function 
$u: \Omega \mapsto \mathbb{R}$, $u\in L^2 (\Omega)$ we consider 
\begin{equation}\label{def:funcional_dirichlet.intro.2257}
\begin{array}{l}
\displaystyle
E_{ii}(u):=\int_{\Omega_\ell} \frac{|\nabla u (x)|^2}{2} {\rm d}x + \frac{1}{2}\int_{\Omega_{n\ell}}\int_{\mathbb{R}^N\setminus \Omega_{\ell} }
J(x-y)(u(y)-u(x))^2\, {\rm d}y{\rm d}x \\[7pt]
\qquad \qquad \quad \displaystyle 
+ \frac{1}{2}\int_{\Omega_{nl}}\int_{\Gamma}G(x,z)\left(u(x)-u(z)\right)^2 {\rm d}\sigma(z) dx
- \int_\Omega f (x) u(x) {\rm d}x.
\end{array}
\end{equation}

Notice that here we are integrating again in the whole $\mathbb{R}^N$, hence, to evaluate this energy we have to take a fixed 
exterior datum 
$
u(x) = 0$, for $x \not\in \Omega.$
As before, we just considered homogeneous Dirichlet conditions. 
Hence, we look for minimizers in the natural set associated with the energy and the Dirichlet exterior 
boundary condition,
$$
H_{ii} = \Big\{ u \in L^2 (\Omega), u|_{\Omega_\ell} =u  \in H^1 (\Omega_\ell), 
 u =0 \mbox{ in } \mathbb{R}^N \setminus \Omega \Big\}.
$$

Our result for this second case reads as follows,

\begin{theorem} \label{teo.1.44}
Assume $(J1)$, $(J2)$, $(1)$, $(2)$, $(G1)$ and $(P2)$.
Given $f \in L^2 (\Omega)$ there exists a unique minimizer $u$ of $E_{ii}$ in $H_{ii}$.
The unique minimizer $u$ is a weak solution to a local equation in $\Omega_\ell$,
\begin{equation}  \label{eq:main.Dirichlet.local.44}
\begin{cases}
\displaystyle  - f(x)=\Delta u (x)  &x\in \Omega_\ell,
\\[7pt]
\displaystyle  u(x)=0,\qquad & x\in \partial \Omega_\ell \cap \partial \Omega, \\[7pt]
\displaystyle \frac{\partial u}{\partial \eta} (x)=  \chi_\Gamma (x)   \int_{\Omega_{nl}} G(x,y)( u(y) - u(x)) {\rm d}y,     & 
x \in \partial \Omega_\ell \cap \Omega ,
\end{cases}
\end{equation}
and a nonlocal one in $\Omega_{n\ell}$,
\begin{equation}  \label{eq:main.Dirichlet.nonlocal.44}
\begin{cases}
\displaystyle  - f(x) = 2 \int_{\Omega_{n\ell}} J(x-y)(u(y)-u(x))\, {\rm d}y
+ \int_{\Gamma} G(y,x)  (u(x)-u(y))\, {\rm d}\sigma(y) \\[7pt]
\displaystyle  \qquad \qquad -
\int_{\mathbb{R}^N\setminus \Omega } J(x-y) u(x) \,{\rm d}y,
\qquad x \in \Omega_{n\ell}, \\[7pt]
\displaystyle u(x)= 0,\qquad  x\in \mathbb{R}^N\setminus \Omega.
\end{cases}
\end{equation}
\end{theorem}

Notice that here the coupling term appears as a flux boundary condition for the local part of the problem.

\begin{remark}
A probabilistic interpretation of this model in terms of particle systems runs as follows:
in the local region $\Omega_{\ell}$ the particles move according to Brownian motion 
and when they arrive to the interface $\Gamma$ they pass to the nonlocal region; in the nonlocal region, as before, we
take an exponential clock that controls the jumps of the particles, 
and when the clock rings a new position is sorted according to the kernel 
$J(x-\cdot)$ (for the movements from $\Omega_{n\ell}$ to $\mathbb{R}^N\setminus \Omega$) or according to $G(\cdot,x)$ 
for jumping back to the local region entering at a point in the hypersurface, $\Gamma$.
The particle gets killed when it arrives to $\mathbb{R}^N \setminus \Omega$.

As before, the minimizer to our functional $E_{ii}(u)$ gives the stationary distribution of particles provided
that there is an external source $f$ (that adds particles where $f>0$ and remove particles where $f<0$). 
\end{remark}

\subsection{Vectorial problems} 

Following the scalar case, we look for functions 
$U:\Omega \to \RR^N,$
that are minimizers of local/nonlocal
energies
restricting our interest to  `elastic' related models. As in the scalar case, one needs
$U$ to be defined outside $\Omega$ (i.e. in the whole  $\mathbb{R}^N$) by imposing a prescribed datum, see below.
For the sake of simplicity  we  consider in detail only
 the quadratic functional associated to the linearized equations of elasticity. Some possible extensions for more general functionals 
 are pointed out in \ref{sec:caso_mas_gral}.

\subsubsection{{\bf First model.}Coupling local/nonlocal elasticity models via source terms}
In this case, we consider the local/nonlocal energy 
\begin{equation}
\label{def:funcional_Dirichlet.intro.Vec}
E_{I}(U):=
  \Xi (U) +
  \frac{1}{2}\int_{\Omega_{n\ell}}\int_{\mathbb{R}^N} 
J(x-y)|(x-y) \cdot (U(y)-U(x))|^2\,{\rm d}y{\rm d}x  
- \int_\Omega F (x) U (x) {\rm d}x,
\end{equation}
with
\begin{equation}
\label{eq:energy_linearized.intro}
\Xi(U):=\mu \int_{\Omega_\ell} |\mathcal{E} (U) (x) |^2 {\rm d}x + \frac{\lambda}{2}  \int_{\Omega_{\ell}} (div (U) (x))^2 {\rm d}x, 
\end{equation}
 and where  $
\LE(U)$ stands for the so-called linearized \emph{strain} tensor
\begin{equation}
 \LE(U)=\frac{\nabla U+\nabla U^T}{2},
 \label{eq:defidestrain}
\end{equation}
and $\mu,\lambda$, are the so-called Lam\'e coefficients (that we assume to be positive).

Since we are integrating in the whole $\mathbb{R}^N$, we have to take a fixed \emph{Dirichlet} exterior datum
$
U(x) = G_D(x)$, for $ x \not\in \Omega.
$
As we did before for the scalar case, we write in detail the  \emph{homogenous} case, i.e. we assume
$G_D \equiv 0.$
In this context, we look for minimizers in the space
$$
H_I = \Big\{ U: U|_{\Omega} \in L^2 (\Omega:\mathbb{R}^N), u|_{\Omega_\ell} \in H^1 (\Omega_\ell:\mathbb{R}^N), 
U = 0 \mbox{ in } \mathbb{R}^N \setminus \Omega \Big\}.
$$

For this model we have the following theorem, where we use the standard notation $\sigma(x)$  for the \emph{linearized}
stress tensor $\sigma=2\mu \mathcal{E}(U)+\lambda div (U)\,Id$ and where  $Div$ denotes the -capitalized- divergence operator acting on matrices in a row-wise sense.

\begin{theorem} \label{teo.elast.88}
Assume $(J1)$, $(J2)$, $(1)$, $(2)$ and $(P1)$.
Given $F \in L^2 (\Omega)$ there exists a unique minimizer $U$ of $E_{I}$ in $H_I$. Moreover, the minimizer $U$ of $E_{I}$ 
in $H_I$ is a weak solution to 
\begin{equation}  \label{eq:main.Dirichlet.local.elast}
\begin{cases}
\displaystyle  F(x)=Div( \sigma (x)) + \int_{\Omega_{n\ell}} J(x-y)[(x-y)\otimes (x-y)] (U(x)-U(y))\,{\rm d} y, &x\in \Omega_\ell,
\\[7pt]
\displaystyle \sigma(x) \eta=0,\qquad x\in \Gamma=\partial \Omega_{\ell} \cap \partial \Omega_{n\ell},&  \\[7pt]
U(x)= 0, \qquad x \in \partial \Omega \cap \partial \Omega_\ell,& 
\end{cases}
\end{equation}
in $\Omega_\ell$ and   
\begin{equation}
\label{eq:main.Dirichlet.nonlocal.elast}
\begin{split}
  F(x)&= \int_{\mathbb{R}^N\setminus \Omega_{n\ell}} J(x-y)\left[(x-y)\otimes (x-y)\right] (U(x)-U(y))\, {\rm d}y
\\
& \qquad +2 \int_{\Omega_{n\ell}} J(x-y)\left[(x-y)\otimes (x-y)\right]  (U(x)-U(y))\, {\rm d}y,
\end{split}
\end{equation}
 in $\Omega_{n\ell}$.
\end{theorem}

\begin{remark}
In classical elasticity theory \cite{Cia,Mar},
\emph{internal forces} arising
between neighboring parts of a solid body are assumed to
act ``locally''  across the common boundary between them. In peridynamics, material points can interact with one another within a certain length scale -usually called a {\em horizon}- through a bond.  Thanks to this approach, it is possible to reflect properties of the microstructure of the body, a capability beyond the classical elasticity theory. In our model, the elastic energy
\eqref{def:funcional_Dirichlet.intro.Vec}, takes into account the interaction between two different materials. We assume that the deformation field associated to one of them (occupying the portion described by $\Omega_\ell$) can be  properly  described by the classical \emph{linearized} elasticity while for the other material (described by $\Omega_{n\ell}$),  a nonlocal (peridynamic) model is better suited. Both parts interact through the     bond modelled by the kernel $J$.

There are basically two different approaches  for finding static or equilibrium states in mechanics. By minimizing an appropriate energy functional or by balancing 
all the involved forces acting on the system. In this case, the explicit form of these forces are visible in the  Euler-Lagrange equations associated to our minimization problem, given in \ref{teo.elast.88}.
\end{remark}

\subsubsection{{\bf Second model.}
Coupling local/nonlocal elasticity models via flux terms}

For 
$U:\Omega \to \RR^N,$
and 
$\Gamma \subseteq \partial \Omega_{\ell} \setminus \partial \Omega,$  we consider the local/nonlocal energy 
\begin{equation}\label{def:funcional_Dirichlet.intro.Vec_II}
\begin{array}{l}
\displaystyle 
E_{II}(u):=
\Xi (U) +
\frac{1}{2}\int_{\Omega_{n\ell}}\int_{\mathbb{R}^N \setminus \Omega_{\ell}} 
J(x-y)|(x-y) \cdot (U(y)-U(x))|^2\,{\rm d}y{\rm d}x \\[7pt]
\displaystyle 
\qquad \qquad + \frac{1}{2}\int_{\Omega_{n\ell}}\int_{\Gamma} 
	G(x,y)|(x-y) \cdot (U(y)-U(x))|^2\,{ {\rm d}}\sigma(y){\rm d}x	- \int_\Omega F (x)U (x) {\rm d}x.
	\end{array}
\end{equation}
Notice that, as in the first model, we are integrating in the whole $\mathbb{R}^N$, therefore to evaluate this energy we have to take a fixed \emph{Dirichlet} exterior datum. To simplify, as we did before, we just consider homogeneous Dirichlet conditions,
$
	U(x) = 0$, for $ x \not\in \Omega$.

For this energy, we look for minimizers in the space
$$
H_{II} = \Big\{U: U|_\Omega \in L^2 (\Omega:\mathbb{R}^N), U|_{\Omega_\ell} \in H^1 (\Omega_\ell:\mathbb{R}^N), 
U = 0 \mbox{ in } \mathbb{R}^N \setminus \Omega \Big\}.
$$

We prove the following result.

\begin{theorem} \label{teo.elast.88-999}
Assume $(J1)$, $(J2)$, $(1)$, $(2)$, $(G1)$ and $(P2)$.
Given $F \in L^2 (\Omega)$, there exists a unique minimizer $U$ of $E_{II}$ in $H_{II}$. The minimizer $U$ is a weak solution to 
the following local-nonlocal system,
\begin{equation}  \label{eq:main.Dirichlet.local.elast_II}
\begin{cases}
\displaystyle  F(x)=Div( \sigma (x)), \qquad x\in \Omega_\ell,
\\[7pt]
\displaystyle \sigma(x) \eta= \int_{\Omega_{n\ell}}
G(y,x)\left[(x-y)\otimes (x-y)\right]  (U(x)-U(y))^T{\rm d}y,\quad x\in \Gamma=\partial \Omega_{\ell} \cap \partial \Omega_{n\ell},&  \\[7pt]
U(x)= 0, \qquad x \in \partial \Omega \cap \partial \Omega_\ell,& 
\end{cases}
\end{equation}
in $\Omega_\ell$  and   
\begin{equation}
\label{eq:main.Dirichlet.nonlocal.elast.567}
\begin{split}
F(x)&= \int_{\mathbb{R}^N\setminus \Omega} J(x-y)\left[(x-y)\otimes (x-y)\right] (U(x)-U(y))\, {\rm d}y
\\
&\quad +2 \int_{\Omega_{n\ell}} J(x-y)\left[(x-y)\otimes (x-y)\right]  (U(x)-U(y))\, {\rm d}y \\
& \quad +\int_{\Gamma}  G(x,y)\left[(x-y)\otimes (x-y)\right]  (U(x)-U(y))\, {{\rm d}}\sigma(y),
\end{split}
\end{equation}
 in $\Omega_{n\ell}$. 
\end{theorem}

\begin{remark}
The energy \eqref{def:funcional_Dirichlet.intro.Vec_II} resembles  \eqref{def:funcional_Dirichlet.intro.Vec}, however the  part of $\Omega_{\ell}$ seen from $\Omega_{n\ell}$ is placed at the boundary and localized in $\Gamma$. Notice that now the interaction
between the local and nonlocal parts of the domain takes place between points in $\Omega_{n\ell}$ and points on $\Gamma$. 
In this case, the bond is given by a different operator encoded by the function 
$G$. 
\end{remark}

\section{The scalar case} \label{sect-scalar}

 Our goal in this section is to look for different energies involving a scalar function 
 that combine local terms acting in $\Omega_\ell$ and nonlocal ones in $\Omega_{n\ell}$ plus a coupling term.

\subsection{{First model} Coupling local/nonlocal problems via source terms}
We study the existence and uniqueness of minimizers to 
the local/nonlocal energy,
\begin{equation}\label{def:funcional_dirichlet.intro}
E_{i} (u):=\int_{\Omega_\ell} \frac{|\nabla u (x)|^2}{2} {\rm d}x + \frac{1}{2}\int_{\Omega_{n\ell}}\int_{\mathbb{R}^N} J(x-y)(u(y)-u(x))^2\, {\rm d}y{\rm d}x
- \int_\Omega f(x) u(x) {\rm d}x,
\end{equation}
in the set $H_{i}$ given by
$$
H_{i} = \Big\{ u \in L^2 (\Omega), u|_{\Omega_\ell} \in H^1 (\Omega_\ell), 
 u = 0 \mbox{ in } \mathbb{R}^N \setminus \Omega \Big\}.
$$

To begin with our analysis of this problem we introduce a couple of useful lemmas. 
 
\begin{lemma}
\label{lema:J0implica_cte}
 Let $D$ be an open  $\delta$ connected set, and $u:D\to \RR$. If
 $$
 \int_{D}\int_{D}J(x-y)(u(x)-u(y))^2 {\rm d}y {\rm d}x=0,
 $$
 then there exists a constant $k \in \mathbb{R}$ such that
 $$u(x)=k, \qquad \mbox{a.e. } x\in D. $$ 
\end{lemma}

\begin{proof}
 Pick $x_0\in D$ and a ball $B_0=B_{\delta}(x_0)$,  we have 
 $$
 C\int_{D\cap B_0 }\int_{D\cap B_0}(u(x)-u(y))^2 {\rm d}y {\rm d}x\le 
 \int_{D\cap B_0 }\int_{D\cap B_0}J(x-y)(u(x)-u(y))^2 {\rm d}y {\rm d}x =0,
 $$
(since $J(x-y)>C$ for $x,y\in B_0$) and  hence  $u(x)=k_0$ a.e. $x\in D\cap B_0$. In order to see that this property holds a.e. $x\in D$, let us introduce the set  $\mathcal{M}=\{A\subset D, A\,\, \mbox{open}: u(x)=k_0 \,\mbox{a.e.} x\in A \}$ with the partial order given by inclusion. Since $\mathcal{M}\neq\emptyset$ there exists a maximal open set $M\in \mathcal{M}$. If $M\subsetneq \Omega$ then we consider the set $\emptyset\neq D\setminus M$. 
 
 If $D\setminus M$ is open we must have -using that $D$ is $\delta$ connected- that $dist(M,D\setminus M)<\delta$. If
 $D\setminus M$ is not open, then $dist(M, D\setminus M)=0$  (since $D$ is open).  
Either case, there exists a ball  $B_1$ of radius $\delta$ such that $B_1\cap D\setminus M\neq \emptyset$ and  $B_1\cap M$ has positive measure (since both, $B_1$ and $M$, are open sets).   Arguing as before we see that $u(x)=k_0$ a.e. $x\in B_1\cap D$, a contradiction (since $M$ is maximal with that property and we would have $M\subsetneq M\cup (B_1\cap D)=M$).
 We see that $M=D$ and the proof is complete. 
\end{proof}

Now we prove a lemma that will be used in what follows.

\begin{lemma}
\label{lema:contagio_escalar}
Let $u_n:\Omega\to \RR$ be a sequence such that $u_n\to 0$ strongly in $L^2(\Omega_{\ell})$ and weakly in $L^2(\Omega_{n\ell})$, if in addition  
\begin{equation}
\label{eq:nolcalomega}
\lim_{n\to\infty}\int_{\Omega_{n\ell}}\int_{\Omega} J(x-y)(u_n(x)-u_n(y))^2 {\rm d}y {\rm d}x =0, 
\end{equation}
 then
 $$
 \lim_{n\to\infty}\int_{\Omega_{n\ell}}|u_n(x)|^2 {\rm d}x =0,
 $$
 that is, we have strong convergence of $u_n$ to zero in $L^2(\Omega_{n\ell})$ and hence in $L^2(\Omega)$.
\end{lemma}

\begin{proof}
From \eqref{eq:nolcalomega},   the convergence of $\{u_n\}$ and property $(J2)$,  we easily find that 
\begin{equation}
 \label{eq:solonl}
\lim_{n\to\infty}\int_{\Omega_{n\ell}}\int_{\Omega_{\ell}} J(x-y)u_n(x)^2 {\rm d}x {\rm d}x=0.
\end{equation}
Let us define  
 $
 A^0_\delta=\{x\in \Omega_{n\ell}: dist(x,\Omega_{\ell})<\delta\}.$
Notice that thanks to property $(P)$ and to the fact that  $\Omega_{n\ell}$ is open we see that  $A^0_\delta$ is open and non empty.  In particular it has positive n-dimensional measure.   For any $x\in \overline{A^0_{\delta}}$
we consider the continuous and \emph{strictly} positive function  $g(x)=|B_{2\delta}(x)\cap \Omega_{\ell}|$. Since 
$\overline{A^0_{\delta}}$ is a compact set, there exists a constant $m>0$ such that $g(x)\ge m$ for any  $x\in \overline{A^0_{\delta}}$. 
 As a consequence 
  $$
  \begin{array}{l}
  \displaystyle 
 \int_{\Omega_{n\ell}}\int_{\Omega_{\ell}} J(x-y)u_n(x)^2 {\rm d}y {\rm d}x \ge 
 \int_{A^0_\delta}\int_{B_{2\delta}(x)\cap \Omega_{\ell}} J(x-y)|u_n(x)|^2 {\rm d}y {\rm d}x  \\[7pt]
 \qquad \displaystyle \ge
 mC\int_{A^0_\delta}|u_n(x)|^2 {\rm d}x.
 \end{array}
 $$ 
 and therefore, thanks to  \eqref{eq:solonl}, $u_n\to 0$ in $L^2(A^0_\delta)$.
 In order to iterate 
 this argument we notice that at this point we know that 
 $u_n\to 0$ strongly in $A^0_\delta$ and weakly in $\Omega_{n\ell} \setminus \overline{A^0_\delta}$, hence again from \eqref{eq:nolcalomega} we get
\begin{equation}
 \label{eq:soloadelta}
 \lim_{n\to\infty}\int_{\Omega_{n\ell} \setminus \overline{A^0_\delta}}\int_{ A^0_\delta} J(x-y)|u_n(x)|^2 {\rm d}y {\rm d}x
 =0.
\end{equation}
Since $\Omega_{n\ell}$ is $\delta$ connected, $dist(\Omega_{n\ell} \setminus \overline{A^0_\delta}, A^0_\delta)<\delta$. Considering now
 $
 A^1_\delta=\{x\in \Omega_{n\ell}\setminus \overline{A^0_\delta}: dist(x,A^0_\delta)<\delta\},
 $
 and proceeding as before, we obtain, from \eqref{eq:soloadelta}, that 
 $u_n\to 0$ strongly in $A^1_{\delta}$. This argument can be repeated and giving strong converge in $L^2(A^j_\delta)$ for 
 $$
 A^j_\delta=\Big\{x\in \Omega_{n\ell}\setminus\overline{ \cup_{0\le i< j} A^i_\delta}: dist(x, \cup_{0\le i< j} A^i_\delta)<\delta \Big\}.
 $$
 Since $\Omega_{n\ell}$ is bounded, we have, for a finite number $K\in \N$, $\Omega_{n\ell}=\cup_{0\le i< K} A^i_\delta$ 
 and therefore  the proof is complete.  \end{proof}

With this lemma at hand we are ready to prove 
the following result that is the key step in order to obtain coerciveness of $E_{i}$.
 
\begin{lemma} \label{lema.1}
There exists a constant $C$ such that
$$
\int_{\Omega_\ell} \frac{|\nabla u (x)|^2 }{2} {\rm d}x + \frac{1}{2}\int_{\Omega_{n\ell}}\int_{\mathbb{R}^N} J(x-y)(u(y)-u(x))^2\, {\rm d}y{\rm d}x
\geq C \int_\Omega |u(x)|^2 {\rm d}x,
$$
for every $u$ in
$H_{i}$.
\end{lemma}

\begin{proof} We argue by contradiction and we assume that there is a sequence $u_n\in H$ such that
$$
 \int_\Omega |u_n(x)|^2 {\rm d}x =1
$$
and 
$$
\int_{\Omega_\ell} \frac{|\nabla u_n (x)|^2}{2} {\rm d}x + \frac{1}{2}\int_{\Omega_{n\ell}}\int_{\mathbb{R}^N} J(x-y)(u_n(y)-u_n(x))^2\, {\rm d}y{\rm d}x
\to 0.
$$

Then, we have that
$$
\int_{\Omega_\ell} \frac{|\nabla u_n (x)|^2}{2} {\rm d}x \to 0 
$$
and 
$$ 
\frac{1}{2}\int_{\Omega_{n\ell}}\int_{\mathbb{R}^N} J(x-y)(u_n(y)-u_n(x))^2\, {\rm d}y{\rm d}x
\to 0.
$$

Since $u_n$ is bounded in $L^2(\Omega_\ell)$ (the integral of $|u_n|^2$ in the whole $\Omega$ is equal to 1) and 
$$
\int_{\Omega_\ell} \frac{|\nabla u_n (x)|^2}{2} {\rm d}x \to 0 
$$
we get that there exists a constant $k_1$ such that, along a subsequence, 
$$
u_n \to k_1
$$
strongly in $H^1(\Omega_\ell)$. 

Now we argue in the nonlocal part $\Omega_{n\ell}$. Since $u_n$ in bounded in $L^2 (\Omega_{n\ell})$ 
we have that (extracting another subsequence if necessary)
$
u_n \rightharpoonup u
$
weakly in $L^2(\Omega_{n\ell})$. Moreover, from 
$$ 
\frac{1}{2}\int_{\Omega_{n\ell}}\int_{\mathbb{R}^N} J(x-y)(u_n(y)-u_n(x))^2\, {\rm d}y{\rm d}x 
\to 0.
$$
we obtain that the limit $u$ verifies 
\begin{equation} \label{gabito}
\begin{array}{l}
\displaystyle 
\frac{1}{2}\int_{\Omega_{n\ell}}\int_{\Omega_{n\ell}} J(x-y)(u(y)-u(x))^2\, {\rm d}y{\rm d}x  \\[7pt]
\qquad \displaystyle \leq \lim_{n \to \infty}\frac{1}{2}\int_{\Omega_{n\ell}}\int_{\Omega_{n\ell}} J(x-y)(u_n(y)-u_n(x))^2\, {\rm d}y{\rm d}x =0
\end{array}
\end{equation}
and 
\begin{equation} \label{eq.777}
\begin{array}{l}
\displaystyle
\frac{1}{2}\int_{\Omega_{n\ell}}\int_{\Omega_{\ell}} J(x-y)(k_1-u(x))^2\, {\rm d}y{\rm d}x \\[7pt]
\qquad \displaystyle  \leq \lim_{n\to \infty}
\frac{1}{2}\int_{\Omega_{n\ell}}\int_{\Omega_{\ell}} J(x-y)(u_n(y)-u_n(x))^2\, {\rm d}y{\rm d}x = 0.
\end{array}
\end{equation}

From the first inequality, \eqref{gabito}, using \ref{lema:J0implica_cte}, we obtain that $u \equiv k_2$ in $\Omega_{n\ell}$ (we use here that we assumed 
that $\Omega_{n\ell}$ is $\delta-$connected). 

From the second inequality, \eqref{eq.777}, we get that
$$
\frac{1}{2}\int_{\Omega_{n\ell}}\int_{\Omega_{\ell}} J(x-y)(k_1-k_2)^2\, {\rm d}y{\rm d}x =0
$$
and hence (assuming condition (3), that says that $\Omega_{n\ell}$ and $\Omega_{\ell}$ are $\delta-$connected between them, that is,
the distance between $\Omega_\ell$ and $\Omega_{n\ell}$ is strictly less than $\delta$ and that $J$ is strictly positive
in the ball $B_{2\delta}$), we must have
$
k_1 = k_2.$

Now, from the fact that $u_n \in H_{i}$ we have that
$
u_n =0 \mbox{ in } \mathbb{R}^N \setminus \Omega.
$
If $\partial \Omega \cap \partial \Omega_\ell \neq \emptyset$ we have that the trace of $u_n$ on
$\partial \Omega \cap \partial \Omega_\ell$ verifies
$
u_n|_{\partial \Omega \cap \partial \Omega_\ell}=0, 
$
and from the strong convergence $u_n \to k_1 $ in $H^1(\Omega_{\ell})$, we conclude that 
$
k_1 = 0
$
and hence
$
k_1 = k_2 = 0.
$
On the other hand, if $\partial \Omega \cap \partial \Omega_\ell = \emptyset$, then we use that
the nonlocal domain is in contact with the exterior of $\Omega$ and thus the nonlocal part of
the energy sees the Dirichlet boundary condition. In this case we have that
$dist(\Omega_{n\ell}, \mathbb{R}^N \setminus \Omega ) =0$. Now, we use that 
$
u_n =0$ in $\mathbb{R}^N \setminus \Omega $
together with
$$ 
\frac{1}{2}\int_{\Omega_{n\ell}}\int_{\mathbb{R}^N \setminus \Omega} J(x-y)(u_n(x))^2\, {\rm d}y{\rm d}x
\to 0,
$$
and that $u_n \to k_2$ weakly in $L^2(\Omega_{n\ell})$, to obtain
$
k_2=0
$
and then we again conclude that
$
k_1 = k_2 = 0
$
in this case.

Up to now we have that
$u_n \to 0 $ strongly in $H^1(\Omega_{\ell})$ and $u_n \to 0$ weakly in $L^2(\Omega_{n\ell})$.
From \ref{lema:contagio_escalar} we obtain that
$$u_n \to 0 \qquad \mbox{ strongly in }L^2(\Omega_{n\ell}).$$

Then, we have
$$
1= \int_\Omega |u_n (x)|^2 {\rm d}x = \int_{\Omega_{\ell}} |u_n (x)|^2 {\rm d}x + \int_{\Omega_{n\ell}} |u_n (x)|^2
{\rm d}x \to 0,
$$
a contradiction that proves the result. 
\end{proof}

\begin{remark} \label{re,.sec.dom} 
As we have mentioned in the introduction, our results hold for more general classes of domains $\Omega_\ell$, $\Omega_{n\ell}$.
In fact $\Omega_\ell$ may have several connected components (that we call $\Omega_\ell^1,..., \Omega^i_\ell$) and also 
$\Omega_{n\ell}$ can have several $\delta-$connected components (called $\Omega_{n\ell}^1,..., \Omega^j_{n\ell}$)
as long as, given two of such components, $A$, $B$, there are $A=C_1,C_2,...,C_l =B \in \{\Omega_\ell^1,..., \Omega^i_\ell, 
\Omega_{n\ell}^1,..., \Omega^j_{n\ell}\} $ with $C_i \subset \Omega_\ell$ for $i$ even
and $C_i \subset \Omega_{n\ell}$ for $i$ odd (that is we are alternating connected components 
of $\Omega_\ell$ and $\delta-$connected components of $\Omega_{n\ell}$) with $dist(C_i,C_{i+1})<\delta$. 

Under this more general condition, repeating the steps used in the previous proof (arguing by contradiction assuming that
there exists a sequence $u_n$ such that the conclusion does not hold), we obtain as limits of $u_n$
different constants $k_1,...,k_{i+j}$ in each of the components $\Omega_\ell^1,..., \Omega^i_\ell, 
\Omega_{n\ell}^1,..., \Omega^j_{n\ell}$. Next, we use that given two components, $A$, $B$, there are 
$A=C_1,C_2,...,C_l =B$ as above (the key here is that $dist(C_i,C_{i+1})<\delta$) to conclude that $k_A=k_{C_{2}} =....= k_B$
and therefore all the constants must be equal $k_1=...=k_{i+j}=k$. Next we use that at least one of the components 
touches the exterior of $\Omega$ to conclude that $k=0$ (we are using here that we have an homogeneous Dirichlet datum
 in $\mathbb{R}^N \setminus \Omega$). From this point the rest of the proof follows exactly as before until we reach a contradiction.
 
\begin{figure}[h]\label{Figura3}
\begin{center}
\includegraphics[scale=.5]{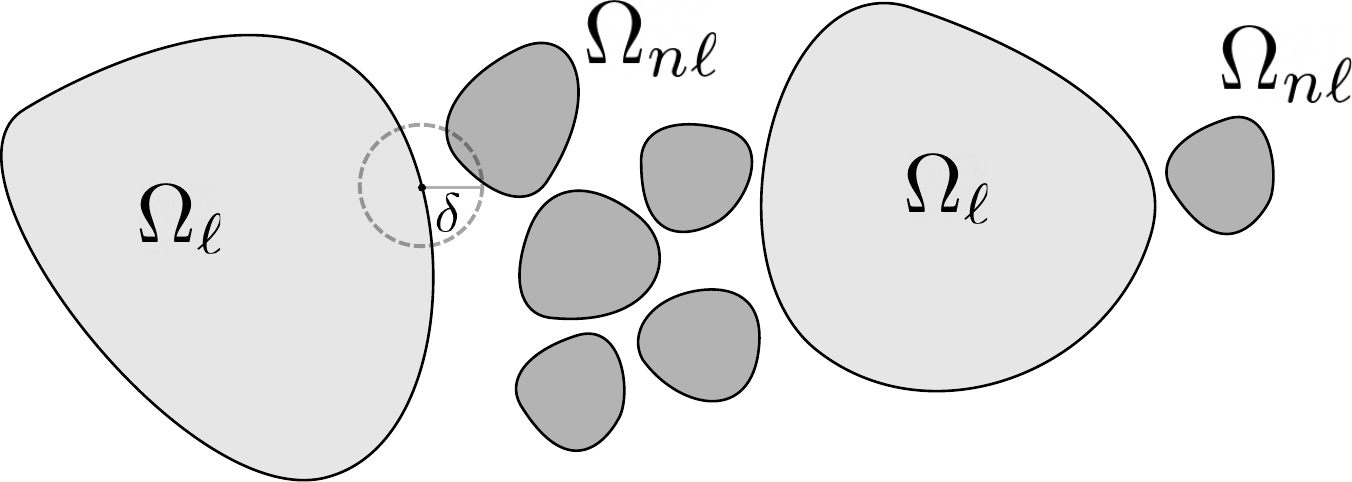}
		\caption{{A possible configuration of $\Omega_\ell$ and $\Omega_{n\ell}$ with several connected components.}}
\end{center}
\end{figure} 
\end{remark}

Now, existence of a unique minimizer follows easily by the direct method of calculus of variations.

\begin{proof}[Proof of \ref{teo.1.22}]
Using \ref{lema.1} and the fact that $f\in L^2(\Omega)$, we have that
$$
E_{i}(u) \geq C \int_{\Omega}u^2 (x) {\rm d}x - \int_\Omega f(x)u(x) {\rm d}x, 
$$
from where it follows that $E_{i}(u)$ is bounded below and coercive. Hence, existence of a minimizer follows
by the direct method of calculus of variations. Just take a minimizing sequence $u_n$ and extract a subsequence
that converges weakly in $L^2 (\Omega)$ and weakly in $H^1 (\Omega_\ell)$. 
Then, we have that the limit remains in $H_{i}$ (since $u_n \in H_i$ 
we obtain that its weak limits verify $u \in L^2 (\Omega)$, $u|_{\Omega_\ell} \in H^1 (\Omega_\ell)$, and 
$u|_{\partial \Omega \cap \partial \Omega_\ell}=0$,  $u = 0 \mbox{ in } \mathbb{R}^N \setminus \Omega$).
Now, we observe that using weak convergence in $L^2 (\Omega)$ we get 
$$
\lim_{n\to \infty} \int_\Omega f (x) u_n (x) {\rm d}x = \int_\Omega f(x) u (x) {\rm d}x
$$ 
and from weak convergence in $L^2 (\Omega)$ and in $H^1 (\Omega_\ell)$, we obtain
$$
\begin{array}{l}
\displaystyle 
\liminf_{n\to \infty} \int_{\Omega_\ell} |\nabla u_n(x) |^2 {\rm d}x + \frac{1}{2 }\int_{\Omega_{n\ell}}\int_{\mathbb{R}^N} J(x-y)(u_n(y)-u_n(x))^2\, {\rm d}y{\rm d}x
\\[7pt] 
\qquad \displaystyle \geq \int_{\Omega_\ell} |\nabla u (x)|^2 {\rm d}x + \frac{1}{2 }\int_{\Omega_{n\ell}}\int_{\mathbb{R}^N} J(x-y)(u(y)-u(x))^2\, {\rm d}y{\rm d}x.
\end{array}
$$
Then, we conclude that $u\in H_{i}$ and that $u$ is a minimizer,
$$E_{i}(u) = \min_{v\in H_{i}} E_{i}(v).$$ 
Uniqueness  follows from the strict convexity of the functional $E_{i}(u)$ in $H_{i}$.

Now, we can easily obtain the associated equation for $u$. 
Let $u$ be the minimizer of $E_{i}$ in $H_{i}$, then for every smooth $\varphi$
with $\varphi =0$ in $\mathbb{R}^N\setminus \Omega$ and every $t\in \mathbb{R}$ we have
$$
E_{i}(u+t \varphi) - E_{i}(u) \geq 0.
$$
Therefore, we have that 
$$
\frac{\partial }{\partial t} E_{i}(u+t\varphi) |_{t=0} =0,
$$
that is,
$$
\displaystyle \int_\Omega f \varphi = \int_{\Omega_\ell} \nabla u \nabla \varphi +
 \int_{\Omega_{n\ell}}\int_{\mathbb{R}^N} J(x-y)(u(y)-u(x)) (\varphi (y) - \varphi(x)) \,{\rm d}y{\rm d}x
$$
Now, we observe that
$$
\begin{array}{l}
\displaystyle 
 \int_{\Omega_{n\ell}}\int_{\mathbb{R}^N} J(x-y)(u(y)-u(x)) (\varphi (y) - \varphi(x)) \,{\rm d}y{\rm d}x
 \\[7pt]
\displaystyle \qquad \qquad 
 = \int_{\Omega_{n\ell}}\int_{\Omega_{n\ell}} J(x-y)(u(y)-u(x)) (\varphi (y) - \varphi(x)) \,{\rm d}y{\rm d}x
 \\[7pt]
\displaystyle \qquad \qquad \qquad 
 +  \int_{\Omega_{n\ell}}\int_{\mathbb{R}^N \setminus \Omega_{n\ell}} J(x-y)(u(y)-u(x)) (\varphi (y) - \varphi(x)) \,{\rm d}y{\rm d}x.
 \end{array}
$$

Using that the kernel is symmetric and Fubini's Theorem we obtain,
$$
\begin{array}{l}
\displaystyle 
\int_{\Omega_{n\ell}}\int_{\Omega_{n\ell}} J(x-y)(u(y)-u(x)) (\varphi (y) - \varphi(x)) \,{\rm d}y{\rm d}x 
 \\[7pt]
\displaystyle \qquad \qquad
=  -2 
 \int_{\Omega_{n\ell}}\int_{\Omega_{n\ell}} J(x-y)(u(y)-u(x))  \,{\rm d}y \, \varphi (x) {\rm d}x.
 \end{array}
$$

On the other hand, we have 
$$
\begin{array}{l}
\displaystyle 
 \int_{\Omega_{n\ell}}\int_{\mathbb{R}^N \setminus \Omega_{n\ell}} J(x-y)(u(y)-u(x)) (\varphi (y) - \varphi(x)) \,{\rm d}y{\rm d}x
  \\[7pt]
\displaystyle \qquad \qquad
= - \int_{\mathbb{R}^N \setminus \Omega_{n\ell }}  \int_{\Omega_{n\ell}} J(x-y)(u(y)-u(x)) \varphi (x)  \,{\rm d}y{\rm d}x
 \\[7pt]
\displaystyle \qquad \qquad \qquad
- \int_{\Omega_{n\ell}}\int_{\mathbb{R}^N \setminus \Omega_{n\ell}} J(x-y)(u(y)-u(x)) \varphi(x) \,{\rm d}y{\rm d}x
\end{array}
$$

Hence, using that $u=\varphi =0$ in $\mathbb{R}^N \setminus \Omega$, we conclude that
$$
\begin{array}{l}
\displaystyle \int_\Omega f \varphi = \int_{\Omega_\ell} \nabla u \nabla \varphi - 
2 \int_{\Omega_{n\ell}}\int_{\Omega_{n\ell}} J(x-y)(u(y)-u(x))  \,{\rm d}y \, \varphi (x) {\rm d}x \\[7pt]
\displaystyle \qquad \qquad - \int_{\Omega_{n\ell}}\int_{\mathbb{R}^N \setminus \Omega_{n\ell}} J(x-y)(u(y)-u(x))  \,{\rm d}y \, \varphi (x) {\rm d}x
\\[7pt]
\displaystyle \qquad \qquad  - \int_{\Omega_{\ell}}\int_{ \Omega_{n\ell} } J(x-y)(u(y)-u(x))  \,{\rm d}y \, \varphi (x){\rm d}x
\end{array}
$$
from where if follows that $u$ is a weak solution to \eqref{eq:main.Dirichlet.local.2277} 
and 
\eqref{eq:main.Dirichlet.nonlocal.2277}.
\end{proof}

\subsection{{Second model} Coupling local/nonlocal problems via flux terms}
Our aim now is to look for a scalar problem with an energy that combines local and nonlocal terms 
acting in different subdomains, $\Omega_\ell$ and $\Omega_{n\ell}$, of $\Omega$, but now the coupling
is made balancing the fluxes.

Recall from the Introduction that we also look at the energy given by 
\begin{equation}\label{def:funcional_dirichlet.intro.79}
\begin{array}{l}
\displaystyle
E_{ii}(u):=\int_{\Omega_\ell} \frac{|\nabla u(x)|^2}{2} {\rm d}x + \frac{1}{2}\int_{\Omega_{n\ell}}\int_{\mathbb{R}^N\setminus \Omega_{\ell} }
J(x-y)(u(y)-u(x))^2\, {\rm d}y{\rm d}x \\[7pt]
\qquad \qquad \quad \displaystyle 
+ \frac{1}{2}\int_{\Omega_{nl}}\int_{\Gamma}G(x,z)\left(u(x)-u(z)\right)^2 {\rm d} \sigma(z) {\rm d}x
- \int_\Omega f(x) u(x) {\rm d}x.
\end{array}
\end{equation}

We take a fixed exterior Dirichlet datum
$
u(x) = 0$, for $ x \not\in \Omega$,
and then we look for minimizers in
$$
H_{ii} = \Big\{ u \in L^2 (\Omega), u|_{\Omega_\ell} =u  \in H^1 (\Omega_\ell), 
 u =0 \mbox{ in } \mathbb{R}^N \setminus \Omega \Big\}.
$$

Again in this case we need a key lemma, similar to \ref{lema.1}, that is needed to obtain coerciveness of the functional. 

\begin{lemma} There exists a constant $C>0$ such that
\begin{equation} \label{cota.clave.22.Dir}
\begin{array}{l}
\displaystyle
\int_{\Omega_\ell} \frac{|\nabla u (x)|^2}{2} {\rm d}x + \frac{1}{2}\int_{\Omega_{n\ell}}\int_{\mathbb{R}^N\setminus \Omega_{\ell} }
J(x-y)(u(y)-u(x))^2\, {\rm d}y{\rm d}x \\[7pt]
\qquad \qquad \quad \displaystyle 
+ \frac{1}{2}\int_{\Omega_{nl}}\int_{\Gamma}G(x,z)\left(u(x)-u(z)\right)^2 {\rm d} \sigma(z) {\rm d}x
\geq  C  \int_{\Omega} |u (x)|^2 {\rm d}x ,
\end{array}
\end{equation}
for every $u \in H_{ii}$.
\end{lemma}

\begin{proof}
We argue by contradiction, then we assume that there is a sequence $w_n\in H_{ii}$ such that
		$$
		\int_\Omega |u_n (x)|^2 {\rm d}x =1
		$$
		and 
		$$
		\begin{array}{l}
\displaystyle \int_{\Omega_\ell} \frac{|\nabla u_n (x)|^2}{2} {\rm d}x 
+ \frac{1}{2}\int_{\Omega_{n\ell}}\int_{\mathbb{R}^N\setminus \Omega_{\ell} } J(x-y)(u_n(y)-u_n(x))^2\, {\rm d}y{\rm d}x \\[7pt]
\qquad \displaystyle 
+ \frac{1}{2}\int_{\Omega_{nl}}\int_{\Gamma}G(x,z)\left(u_n(x)-u_n(z)\right)^2 {\rm d}\sigma(z) {\rm d}x
		\to 0.
		\end{array}
		$$
		Then we have that
		$$
		\int_{\Omega_\ell} |\nabla u_n (x)|^2 {\rm d}x \to 0, 
		$$
		$$
	     \int_{\Omega_{nl}}\int_{\Gamma}G(x,z)\left(u_n(x)-u_n(z)\right)^2 {\rm d}\sigma(z) {\rm d}x \to 0,
		$$
		and
		$$ 
       \int_{\Omega_{n\ell}}\int_{\mathbb{R}^N\setminus \Omega_{\ell} } J(x-y)(u_n(y)-u_n(x))^2\, {\rm d}y{\rm d}x
		\to 0.
		$$
		
		Since 
		$$
		\int_\Omega |u_n(x)|^2 {\rm d}x = \int_{\Omega_\ell} |u_n (x)|^2 {\rm d}x + \int_{\Omega_{n\ell}} |u_n (x)|^2 {\rm d}x =1
		$$
 we obtain that $u_n|_{\Omega_\ell}$ is bounded in $L^2(\Omega_\ell)$ and then, using that, 
		$$
		\int_{\Omega_\ell} \frac{|\nabla u_n (x)|^2}{2} {\rm d}x \to 0 
		$$
		we get that there exists a constant $k_1$, such that, along a subsequence, 
		$
		u_n \to k_1
		$
		strongly in $H^1(\Omega_\ell)$. Hence, using the trace theorem in $\Gamma$ 
		we obtain
		$
		u_n \to k_1
		$
		strongly $L^2(\Gamma)$.

		Now we argue in the nonlocal part $\Omega_{n\ell}$. Since $u_n|_{\Omega_{n\ell}}$ is bounded in $L^2 (\Omega_{n\ell})$ 
		we have that (extracting another subsequence if necessary)
		$
		u_n \rightharpoonup u
		$
		weakly in $L^2(\Omega_{n\ell})$. Now, we have that
		$$ 
		\int_{\Omega_{n\ell}}\int_{\mathbb{R}^N\setminus \Omega_{\ell} } J(x-y)(u_n(y)-u_n(x))^2\, {\rm d}y{\rm d}x \to 0
		$$
		 and therefore, 
		 $$ 
		\int_{\Omega_{n\ell}}\int_{\Omega_{n\ell} } J(x-y)(u_n(y)-u_n(x))^2\, {\rm d}y{\rm d}x \to 0.
		$$
		Hence, the weak limit satisfies, 
		$$
		\int_{\Omega_{n\ell}} \!\!\int_{\Omega_{n\ell}} J(x-y)(u(y)-u(x))^2{\rm d}y{\rm d}x \!
		\leq \! \lim_{n \to \infty}\!\! \frac{1}{2} \!\!\int_{\Omega_{n\ell}} \!\! \int_{\Omega_{n\ell}} J(x-y)(u_n(y)-u_n(x))^2 {\rm d}y{\rm d}x =0.
		$$
		Hence, we get that $u \equiv k_2$ in $\Omega_{n\ell}$ (again here we use here that $\Omega_{n\ell}$ is $\delta-$connected).
	
From the weak convergence of $u_n$ to $u$ in
 $L^2(\Omega_{n\ell})$ and the strong convergence of $u_n$ to $u$ in
 $L^2(\Gamma)$ we get
$$
\int_{\Omega_{nl}}\int_{\Gamma}G(x,z)\left(k_2-k_1\right)^2 {\rm d}\sigma(z) {\rm d}x =0.
$$
		
Hence, from the fact that
$$
\int_{\Omega_{nl}}\int_{\Gamma}G(x,z) {\rm d}\sigma(z) {\rm d}x >0,
$$ 
we obtain 
$
k_1 = k_2.
$

Now, from the fact that $u_n \in H_{ii}$ 
if $\partial \Omega \cap \partial \Omega_\ell \neq \emptyset$ we have that the trace of $u_n$ on
$\partial \Omega \cap \partial \Omega_\ell$ verifies
$
u_n|_{\partial \Omega \cap \partial \Omega_\ell}=0, 
$
and from the strong convergence $u_n \to k_1 $ strongly in $H^1(\Omega_{\ell})$, we conclude that 
$
k_1 = 0
$
and hence
$
k_1 = k_2 = 0.
$
On the other hand, if $\partial \Omega \cap \partial \Omega_\ell = \emptyset$, then, as we did before, 
we use that 
$
u_n =0$ in $ \mathbb{R}^N \setminus \Omega $
together with
$$ 
\frac{1}{2}\int_{\Omega_{n\ell}}\int_{\mathbb{R}^N \setminus \Omega} J(x-y)(u_n(x))^2\, {\rm d}y{\rm d}x
\to 0,
$$
and that $u_n \to k_2$ weakly in $L^2(\Omega_{n\ell})$, to obtain
$
k_2=0
$
and then we again conclude that
$
k_1 = k_2 = 0,
$
in this case.
		
		Up to now we have that
		$u_n \to 0 $ strongly in $H^1(\Omega_{\ell})$ and $v_n \to 0$ weakly in $L^2(\Omega_{n\ell})$.
		Then, as
		$$
		1= \int_\Omega |u_n(x)|^2 {\rm d}x = \int_{\Omega_{\ell}} |u_n (x)|^2 {\rm d}x + \int_{\Omega_{n\ell}} |u_n (x)|^2 {\rm d}x
		$$
		we get that
		$$
		\int_{\Omega_{n\ell}} |u_n (x)|^2 {\rm d}x \to 1.
		$$
		
		Now, we need a result like \ref{lema:contagio_escalar} to conclude, but notice that here we have to 
		adapt the argument since we want to propagate the strong convergence to zero in $L^2$ from $\Gamma$
		to the nonlocal region $\Omega_{n\ell}$. To this end, we go back to 
		$$
\int_{\Omega_{nl}}\int_{\Gamma}G(x,z)\left(u_n(x)-u_n(z)\right)^2 {\rm d}\sigma(z) {\rm d}x 
		\to 0.
		$$
		and we obtain
		$$
		\begin{array}{l}
		\displaystyle 0 = \lim_{n \to \infty}
		\int_{\Omega_{nl}}\int_{\Gamma}G(x,z)\left(u_n(x)-u_n(z)\right)^2 {\rm d}\sigma(z) {\rm d}x \\[7pt]
		\qquad \displaystyle
		= \lim_{n \to \infty}
		\int_{\Omega_{nl}}\int_{\Gamma}G(x,z)|u_n(x)|^2 {\rm d}\sigma(z) {\rm d}x
		\\[7pt]
		\qquad \displaystyle
		\qquad +  \lim_{n \to \infty} \int_{\Omega_{nl}}\int_{\Gamma}G(x,z)|u_n(z)|^2 {\rm d}\sigma(z) {\rm d}x
		\\[7pt]
		\qquad \displaystyle
		\qquad -  \lim_{n \to \infty} 2 \int_{\Omega_{nl}}\int_{\Gamma}G(x,z)u_n(x)u_n(z) {\rm d}\sigma(z) {\rm d}x.
		\end{array}
		$$
		Since $u_n \to 0$ strongly in $L^2(\Gamma)$, we have that 
		$$
		\lim_{n \to \infty}
		 \int_{\Omega_{nl}}\int_{\Gamma}G(x,z)|u_n(z)|^2 {\rm d}\sigma(z) {\rm d}x = 0
		$$
		and 
		$$
		\lim_{n \to \infty} \int_{\Omega_{nl}}\int_{\Gamma}G(x,z)u_n(x)u_n(z) {\rm d}\sigma(z) {\rm d}x = 0.
		$$
		Therefore, we obtain that 
		\begin{equation}\label{acosta1.22}
		\lim_{n \to \infty} \int_{\Omega_{n\ell}} |u_n(x)|^2 \int_{\Gamma} G(x,z)  \, {\rm d}\sigma(z) {\rm d}x = 0.
		\end{equation}
		Let $
		B_1 = \{ x \in \Omega_{n\ell} :  \int_{\Gamma} G(x,z)  \, {\rm d}\sigma(z) >0 \}.
		$
		We have that $u_n \to 0$ strongly in $L^2(B_1)$ 
		Notice that $|B_1|>0$.
		Now, take 
		$
		B_2 = \{ x \in \Omega_{n\ell} :  dist(x, B_1) <\delta \}.
		$
		We have
		$$
		\begin{array}{l}
		\displaystyle \lim_{n \to \infty}
		\int_{B_1}\int_{B_2} J(x-y)(u_n(y)-u_n(x))^2\, {\rm d}y{\rm d}x \\[7pt]
		\qquad \displaystyle
		\leq \lim_{n \to \infty} \int_{\Omega_{n\ell}}\int_{\Omega_{n\ell}} J(x-y) |u_n(x)-u_n (y)|^2 \, {\rm d}y{\rm d}x =0
		\end{array}
		$$
		Hence, we get
		$$
		\begin{array}{l}
		\displaystyle 0=\lim_{n \to \infty}
		\int_{B_1}\int_{B_2} J(x-y)(u_n(y)-u_n(x))^2\, {\rm d}y{\rm d}x \\[7pt]
		\qquad \displaystyle
		= \lim_{n \to \infty}\int_{B_1}\int_{B_2} J(x-y) |u_n(x)|^2-u_n (y)|^2 \, {\rm d}y{\rm d}x \\[7pt]
		\qquad \qquad \displaystyle
		+ \lim_{n \to \infty}\int_{B_1}\int_{B_2} J(x-y) u_n(x) u_n (y) \, {\rm d}y{\rm d}x \\[7pt]
		\qquad \qquad \displaystyle
		+ \lim_{n \to \infty}\int_{B_2} |u_n (y)|^2 \int_{B_1}J(x-y)  \, {\rm d}x \, {\rm d}y.
		\end{array}
		$$
		As $u_n \to 0$ strongly in $L^2(B_1)$ and $\int_{B_1}J(x-y)  \, {\rm d}x \geq c > 0$ for every 
		$y \in B_2$, we conclude that $u_n \to 0$ strongly in $L^2(B_2)$.
		Iterating this procedure a finite number of times we obtain that 
		 $u_n \to 0$ strongly in $L^2 (\Omega_{n\ell})$,
		a contradiction with the fact that
		$$
		\int_{\Omega_{n\ell}} |u_n (x)|^2 {\rm d}x \to 1.
		$$
		This ends the proof.
		\end{proof}

\begin{proof}[Proof of \ref{teo.1.44}] 
The proof is similar to the proof of \ref{teo.1.22}.
The functional $E_{ii}(u)$ is  weakly semicontinuous, bounded below and coercive in $H_{ii}$ that is
weakly closed. Hence, existence of a minimizer follows
by the direct method of calculus of variations and its uniqueness follows from the strict convexity.

Finally, we obtain the associated equation with Dirichlet boundary conditions.
Let $u$ be the minimizer of $E_{ii}(u)$ in $H_{ii}$, then for every smooth $\varphi$
in $H_{ii}$ and every $t\in \mathbb{R}$ we have
$$
E_{ii}(u+t \varphi) - E_{ii}
(u) \geq 0.
$$
Therefore, we get that 
$$
\frac{\partial }{\partial t} E_{ii}(u+t\varphi) |_{t=0} =0,
$$
that is,
$$
\begin{array}{l}
\displaystyle \int_\Omega f \varphi = \int_{\Omega_\ell} \nabla u \nabla \varphi +
 \int_{\Omega_{n\ell}}\int_{\mathbb{R}^N\setminus \Omega_{\ell}} J(x-y)(u(y)-u(x)) (\varphi (y) -\varphi(x)) \,{\rm d}y{\rm d}x
 \\[7pt]
\displaystyle \qquad \qquad  + \int_{\Omega_{n\ell}}\int_{\Gamma} G(x,y)(u(y)-u(x)) \varphi (y) \, {\rm d} \sigma(y) {\rm d}x
\\[7pt]
\displaystyle \qquad \qquad  - \int_{\Omega_{n\ell}}\int_{\Gamma} G(x,y)(u(y)-u(x)) \varphi (x) \, {\rm d} \sigma(y) {\rm d}x.
\end{array}
$$
Using that the kernel $J$ is symmetric we obtain
$$
\begin{array}{l}
\displaystyle \int_\Omega f \varphi = \int_{\Omega_\ell} \nabla u \nabla \varphi - 2
 \int_{\Omega_{n\ell}}\int_{\Omega_{n\ell}} J(x-y)(u(y)-u(x))  \,{\rm d}y\, \varphi(x) {\rm d}x \\[7pt]
\displaystyle \qquad \qquad -
 \int_{\Omega_{n\ell}}\int_{\mathbb{R}^N\setminus \Omega } J(x-y)(u(y)-u(x)) \varphi(x) \,{\rm d}y{\rm d}x
 \\[7pt]
\displaystyle \qquad \qquad  - \int_{\Gamma} \int_{\Omega_{n\ell}} G(y,x)(u(y)-u(x))  {\rm d}y \varphi (x) \, {\rm d} \sigma(x)
\\[7pt]
\displaystyle \qquad \qquad  - \int_{\Omega_{n\ell}}\int_{\Gamma} G(x,y)(u(y)-u(x))  \, {\rm d} \sigma(y)\, \varphi (x){\rm d}x,
\end{array}
$$
that is, $u$ is a weak solution to \eqref{eq:main.Dirichlet.local.44} and 
\eqref{eq:main.Dirichlet.nonlocal.44}.
\end{proof}

\section{The vectorial case} \label{sect-vectorial}

We begin this section by recalling a classical result in  elasticity.  From  definition  \eqref{eq:defidestrain},  we see that   
\begin{equation}
\label{eq:gradepsilon}
 \|\LE(U)\|_{L^2(\Omega)^{N\times N}}\le C \| \nabla U\|_{L^2(\Omega)^{N\times N}},
\end{equation}
and since  the set of \emph{infinitesimal rigid movements}, defined as
$$
RM:=\Big\{R(x)=
Mx+p,\,
\mbox{with}\, M\in \RR^{N\times N}, 
M=-M^T \, \mbox{and}\, p\in \RR^N \Big\},
$$
is clearly included in the kernel of $\LE$, \eqref{eq:gradepsilon} can not be reversed in general.  Nonetheless,  Korn's inequalities assert  that under appropriate assumptions $\LE(U)$ can play the role of $\nabla U$. In particular, the following results hold for $\Omega\subset \RR^N$ belonging to a rather general class of domains (in particular for Lipschitz domains), see \cite{AD} for the proofs. It holds that 
$$ 
H^1(\Omega)^N\equiv 
\Big\{U\in L^2(\Omega)^N:\LE(U)\in  L^2(\Omega)^{N\times N} \Big\},
$$
which amounts to
\begin{equation}
 \label{eq:eqiv_H1}
\|U\|_{L^2(\Omega,\RR^{N})}+\|\nabla U\|_{L^2(\Omega,\RR^{N\times N})}
\sim
\|U\|_{L^2(\Omega,\RR^{N})}
+\|\LE(U)\|_{L^2(\Omega,\RR^{N\times N})}
 \end{equation}
and the following inequality 
\begin{equation}
 \label{eq:korn_cociente}
\inf_{R\in RM}\|U-R\|_{H^1(\Omega,\RR^{N\times N})}\le C\|\LE(U)\|_{L^2(\Omega,\RR^{N\times N})}  
\end{equation}
for a constant $C=C(\Omega)$ (that gives in particular a reversed version of \eqref{eq:gradepsilon}).

An immediate consequence of \eqref{eq:korn_cociente} is the equivalence
$$
U\in RM \Leftrightarrow \LE(U)=0. 
$$
The nonlocal counterpart enjoys a similar property that plays -in the vectorial setting- the role of \ref{lema:J0implica_cte} in the scalar case. We state it as \ref{lema:nucleodelepsilon} (see also Lemma 2.4 in \cite{MenDu}).

\begin{lemma} 
\label{eq:caractRigid}
Assume that 
\begin{equation}
 \label{eq:nolocal-epsilon}
\int_{B_{r_0}(z_0)}\int_{B_{r_0}(z_0)} 
|(x-y) \cdot (U(x)-U(y))|^2\,{\rm d}y{\rm d}x=0 
\end{equation}
for some ball $B_{r_0}(z_0)\subset   \RR^N$ and $U\in L^2(\Omega,\RR^n)$, then $$U(x)=Mx+b$$ with $b\in \RR^N$ and  $M\in \RR^{N\times N}$, $M^T=-M$.
\end{lemma}

\begin{proof}
First we show that $U$ is a continuous function (it has a continuous representative). Take $0<r<r_0$ and $\rho<r_0-r$.  For  any $y_0\in B_r(z_0)$, \eqref{eq:nolocal-epsilon} and Fubini's Theorem says that
$$
\strokedint_{B_\rho(y_0)} (x-y) \cdot (U(x)-U(y))\, {\rm d}y=\frac{1}{|B_\rho(y_0)|}\int_{B_\rho(y_0)}(x-y) \cdot (U(x)-U(y))\, {\rm d}y=0,
$$
a.e. $x\in B_r(z_0)$.
Calling $\bar U=\strokedint_{B_\rho(y_0)}U(y)\, {\rm d}y$ and using that $\strokedint_{B_\rho(y_0)}y\, {\rm d}y=y_0$, we have
$$
\begin{array}{l} 
\displaystyle 
\strokedint_{B_\rho(y_0)} (x-y) \cdot (U(x)-U(y))\, {\rm d}y  
 =
x\cdot( U(x)- \bar U) -y_0\cdot U(x) + \strokedint_{B_\rho(y_0)}y\cdot U(y)\, {\rm d}y \\[7pt]
\displaystyle \qquad =
(x-y_0)\cdot( U(x)- \bar U)
+\strokedint_{B_\rho(y_0)}y\cdot (U(y)-\bar{U})\, {\rm d}y
\end{array}
$$
which says that for any $y_0\in B_r(z_0)$
$$
0=(x-y_0)\cdot (U(x)-\bar U)+\strokedint_{B_\rho(y_0)}y(U(y)-\bar U)\, {\rm d}y
$$
a.e. $x\in B_r(z_0)$.
Using again \eqref{eq:nolocal-epsilon} and Fubini -with a variable called $y_0$ instead of $y$- we see  that 
$$
0=\strokedint_{B_\rho(x_0)}(x-y_0) \cdot (U(x)-U(y_0))\, {\rm d}x,
$$
a.e. 
$y_0 \in B_r(z_0)$ and any $x_0\in B_r(z_0)$.
Therefore, invoking the last two identities we get
$$
\begin{array}{rl}
\displaystyle 
0= & \displaystyle \strokedint_{B_\rho(x_0)}\left((x-y_0)\cdot (U(x)-\bar U)+\strokedint_{B_\rho(y_0)}y(U(y)-\bar U)\, {\rm d}y \,\right)\, {\rm d}x
\\[7pt]
& \displaystyle 
=\strokedint_{B_\rho(x_0)}(x-y_0)\cdot (U(x)-U(y_0)+U(y_0)-\bar U)\, {\rm d}x+\strokedint_{B_\rho(y_0)}y(U(y)-\bar U)\,{\rm d}y 
\\[7pt]
& \displaystyle
=\strokedint_{B_\rho(x_0)}(x-y_0)\cdot (U(y_0)-\bar U)\, {\rm d}x+\strokedint_{B_\rho(y_0)}y(U(y)-\bar U)\, {\rm d}y  
\\[7pt]
& \displaystyle
=(x_0-y_0)\cdot (U(y_0)-\bar U) +\strokedint_{B_\rho(y_0)}y(U(y)-\bar U)\, {\rm d}y  ,
\end{array}
$$
for \emph{any} $x_0\in B_r(z_0)$ and a.e. $y_0\in B_r(z_0)$. 
Since $(U(y_0)-\bar U)$ and $\strokedint_{B_\rho(y_0)}y(U(y)-\bar U)\,dy$ do not depend on $x_0$ we see that
$$
U(y_0)=\bar U=\strokedint_{B_\rho(y_0)}U(y)\, {\rm d}y
$$
a.e. $y_0\in B_r(z_0)$.
This implies that $U$ has a continuous representative in  $B_r(z_0)$ (and hence in $B_{r_0}(z_0)$ since $r<r_0$ is arbitrary). Indeed, fix $\rho$ and define 
$$W(x)=\strokedint_{B_\rho(x)}U (y)\, {\rm d}y,$$ clearly $W$ is continuous and $W=U$ a.e. $x_0\in B_\rho(z_0)$. Moreover, as a byproduct of our arguments 
it turns out that each coordinate of $W$ is a harmonic function.

Assuming continuity of $U$, \eqref{eq:nolocal-epsilon} says that
\begin{equation}
\label{eq:clasica=0}
   (x-y) \cdot (U(x)-U(y))=0 
\end{equation}
for \emph{any} pair $x,y\in B_{r_0}(z_0)$ and then the Lemma can be proved  by means of well-known arguments (see for instance Prop. 1.2 in \cite{TeMir}). 
For the sake of completeness here we give a very  short proof taking advantage of the extra regularity of $U$. Indeed, 
differentiating \eqref{eq:clasica=0} w.r.t. $x$ and evaluating in $z_0$ gives
$
DU(z_0)(y-z_0)=U(y)-U(z_0),
$ 
for any $y\in B_{r_0}(z_0)$
and then
$U(y)$ agrees with a linear function
in $B_{r_0}(z_0)$. 
Moreover, 
previous identity and \eqref{eq:clasica=0} gives
$
(y-z_0)^tDU(z_0)(y-z_0)=(y-z_0)\cdot(U(y)-U(z_0))=0,
$
and hence
$
DU(z_0)=-DU(z_0)^t.
$
Calling $M=DU(z_0)$ the lemma follows. 
\end{proof}


\begin{lemma}
 \label{lema:nucleodelepsilon}
 Let $D$ be an open $\delta$ connected set, then
\begin{equation}
 \label{eq:rm-debil}
U\in RM \Leftrightarrow
\int_{D}\int_{D} 
J(x-y)|(x-y) \cdot (U(x)-U(y))|^2\,{\rm d}y{\rm d}x=0.
 \end{equation}
 \end{lemma}
 
\begin{proof} 
The implication $\Rightarrow$ is immediate. For the other implication, let us take $z_0\in \Omega$ and a ball $B_{r_0}(z_0)$
such that $r_0=\min\{\delta/2,dist(z_0,\partial \Omega) \}$, we have
$$
\int_{B_{r_0}(z_0)}\int_{B_{r_0}(z_0)} 
J(x-y)|(x-y) \cdot (U(x)-U(y))|^2\,{\rm d}y{\rm d}x=0
$$
and since 
$J>C$ for $(x,y)\in B_{r_0}(z_0)\times B_{r_0}(z_0)$, \ref{eq:caractRigid}, says  that there exists $R_0\in RM$  such that 
$R_0=U(x)$ a.e. $x\in B_{r_0}(z_0)$. Now the proof follows as in \ref{lema:J0implica_cte}, we include the details for
completeness.
Let $\mathcal{M}=\{A\subset \Omega, A\, \mbox{open}: U(x)=R_0 \,\mbox{a.e.} x\in A \}$ with the partial order given by $\subset$. Since $\mathcal{M}\neq\emptyset$ there exists a maximal open set $M\in \mathcal{M}$. If $M\subsetneq D$, since $D\setminus M$ is not open ($\Omega$ connected)  there exists $z_1\in D \setminus M$ and a ball $B_{r_1}(z_1)$,$r_1=\min\{\delta/2,dist(z_1,\partial D) \}$ such that $B_{r_1}(z_1)\cap M\neq \emptyset$ and arguing as before
$U(x)=R_1\in RM$ a.e. $x\in B_{r_1}(z_1)$, hence $R_1=R_0$ since
$B_{r_1}(z_1)\cap M$ has positive measure (is open).
Since $M$ is maximal we see that neccessarily $M=D$. 
 \end{proof}

\subsection{First model. Coupling local/nonlocal elasticity models via source terms}

Recall that we aim to study the local/nonlocal energy 
\begin{equation}
\label{def:funcional_Dirichlet.intro.Vec.55}
E_{I}(U):=
  \Xi (U) +
  \frac{1}{2}\int_{\Omega_{n\ell}}\int_{\mathbb{R}^N} 
J(x-y)|(x-y) \cdot (U(y)-U(x))|^2\,{\rm d}y{\rm d}x  
- \int_\Omega F (x) U (x) {\rm d}x,
\end{equation}
with
\begin{equation}
\label{eq:energy_linearized.intro.77}
\Xi(U):=\mu \int_{\Omega_\ell} |\mathcal{E} (U) (x) |^2 {\rm d}x + \frac{\lambda}{2}  \int_{\Omega_{\ell}} (div (U) (x))^2 {\rm d}x, 
\end{equation}
where  $
\LE(U)$ stands for 
\begin{equation}
 \LE(U)=\frac{\nabla U+\nabla U^T}{2}.
 \label{eq:defidestrain.99}
\end{equation}
and $\mu$ and $\lambda$ are two positive coefficients.

In this context, we look for minimizers in the space
$$
H_I = \Big\{ U: U|_{\Omega} \in L^2 (\Omega:\mathbb{R}^N), u|_{\Omega_\ell} \in H^1 (\Omega_\ell:\mathbb{R}^N), 
  U = 0 \mbox{ in } \mathbb{R}^N \setminus \Omega \Big\}.
$$

We will need the following substitute of \ref{lema:contagio_escalar}.

\begin{lemma}
\label{lema:contagio_vectorial}
Let $U_n:\Omega\to \RR^N$ be a sequence such that $U_n\to 0$ strongly in $L^2(\Omega_{\ell})$ and weakly in $L^2(\Omega_{n\ell})$, if in addition  
\begin{equation}
\label{eq:nolcalomega_vect}
\lim_{n\to\infty}\int_{\Omega_{n\ell}}\int_{\Omega} J(x-y)((x-y)\cdot (U_n(x)-U_n(y)))^2 {\rm d}y {\rm d}x=0 
\end{equation}
 then
 $$
 \lim_{n\to\infty}\int_{\Omega_{n\ell}}|U_n(x)|^2 {\rm d}x=0.
 $$
\end{lemma}

\begin{proof}
We sketch the proof, that follows the lines of that of \ref{lema:contagio_escalar}, paying attention only to  the main differences.
From \eqref{eq:nolcalomega_vect},   the convergence of $\{u_n\}$ and property (J2),  we  find the following substitute of \eqref{eq:solonl}
\begin{equation}
 \label{eq:solonl_vect}
\int_{\Omega_{n\ell}}\int_{\Omega_{\ell}} J(x-y)|(x-y) \cdot U_n(x)|^2\, {\rm d}y{\rm d}x
\to 0.
\end{equation}
Let us define,  
 $
 A^0_\delta=\{x\in \Omega_{n\ell}: dist(x,\Omega_{\ell})<\delta \}.
 $
and arguing as in \ref{lema:contagio_escalar}, we see that
the continuous function $g(x):A^0_{\delta}\to \RR$,  $g(x)=|B_{2\delta}(x)\cap \Omega_{\ell}|$ verifies  $g(x)\ge m>0$ for any  $x\in \overline{A^0_{\delta}}$. 
 As a consequence we have
  $$
  \begin{array}{l}
  \displaystyle 
 \int_{\Omega_{n\ell}}\int_{\Omega_{\ell}} J(x-y)\left((x-y)\cdot U_n(x)\right)^2 {\rm d}y {\rm d}x  \\[7pt]
 \qquad \displaystyle \ge 
 \int_{A^0_\delta}\int_{B_{2\delta}(x)\cap \Omega_{\ell}} J(x-y)\left((x-y)\cdot U_n(x)\right)^2 {\rm d}y {\rm d}x. 
 \end{array}
 $$
 In order to proceed we need to bound by below the scalar product. Since $g(x)\ge m>0$ there exists  a fixed $r>0$ such that for any  $x\in A^0_\delta$ we can find a ball $B_r(y_x)$ centered at certain $y_x\in \Omega_{\ell}$, 
 such that $B_r(y_x)\subset B_{2\delta}(x)\cap \Omega_{\ell}$. Therefore, 
  $$
  \begin{array}{l} \displaystyle 
  \int_{A^0_\delta}\int_{B_{2\delta}(x)\cap \Omega_{\ell}} J(x-y)\left((x-y)\cdot U_n(x)\right)^2 {\rm d}y {\rm d}x 
  \\[7pt]
 \qquad \displaystyle \ge C 
 \int_{A^0_\delta}\int_{B_{r/2}(y_x)} \left((y-x)\cdot U_n(x)\right)^2 {\rm d}y {\rm d}x, 
 \end{array} $$
where we have also used $(J1)$. Since $y\in B_{r/2}(y_x)$ and $x\notin B_{r}\subset \Omega_{\ell}$, we have $\|x-y\|\ge r/2$. 
Calling  $C_x$, the smaller  cone  at $x$ containing
$B_{r/2}(y_x)$. Notice that $C_x$ has an opening angle $\alpha$ bounded by below by a constant independent of $x$, since $tg(\alpha)\ge \frac{r}{4\delta}>0$.
Let us call  $C_u$, the cone at $x$ with axis $U_n(x)$ and opening $\frac{\pi}{2}-\alpha/4$.
Clearly $C_x\cap C_u$ contains a cone $C_I$ with an opening angle $\beta\ge \alpha/2$ and therefore there exists a constant $c=c(\alpha)>0$  such that
$
|C_I\cap B_{r/2}(y_x)|\ge  cr^N.
$
On the other hand, for any element $z\in C_I\subset C_u$, $ang(z,U_n(x))\le \frac{\pi}{2} -\frac{\alpha}{4}$. Taking into account that for any $y\in B_{y_x}$, $y-x\in C_x$ we see that
  $$
 \int_{B_{r/2}(y_x)} \left((y-x)\cdot U_n(x)\right)^2 {\rm d}y \ge c r^{N+2} \|U_n(x)\|^2\cos^2\left(\pi-\frac{\alpha}{4}\right)
 $$
 since $|(y-x)\cdot U_n(x)|\ge\frac{r}{2} \|U_n(x)\|\cos\left(\pi-\frac{\alpha}{4}\right).$ Therefore,
$$
\begin{array}{l}
\displaystyle 
C 
\int_{A^0_\delta}\int_{B_{r/2}(y_x)} \left((y-x)\cdot U_n(x)\right)^2 {\rm d}y {\rm d}x \\[7pt]
\qquad \displaystyle \ge 
 Cc r^{N+2} \|U_n(x)\|^2\cos^2\left(\pi-\frac{\alpha}{4}\right) \int_{A^0_\delta}\|U_n(x)\|^2 {\rm d}x
 \end{array} $$
showing, thanks to \eqref{eq:solonl_vect}, that $U_n\to 0$ strongly in $L^2(A^0_\delta)$. Using similar arguments, the proof can be concluded following the  steps developed in \ref{lema:contagio_escalar}.
\end{proof}

Now we are ready to show the key property needed to obtain coerciveness of the functional.

\begin{lemma} \label{lema.1.99}
There exists a constant $C$ such that
$$
\Xi(U) + \frac{1}{2}\int_{\Omega_{n\ell}}\int_{\mathbb{R}^N} 
J(x-y)|(x-y) \cdot (U(y)-U(x))|^2\,{\rm d}y{\rm d}x
\geq C \int_\Omega |U|^2,
$$
for every $U$ in
$$
H_{I} = \Big\{ U \in L^2 (\Omega), U|_{\Omega_\ell} \in H^1 (\Omega_\ell), U|_{\partial \Omega \cap \partial \Omega_\ell}=0|_{\partial \Omega \cap \partial \Omega_\ell},  U =0 \mbox{ in } \mathbb{R}^N \setminus \Omega \Big\}.
$$
\end{lemma}

\begin{proof} Since
$$
\Xi(U)\ge \int_{\Omega_\ell} \frac{|\mathcal{E} (U)|^2}{2},
$$
we can argue by contradiction. Assume that there is a sequence $U_n\in H$ such that
$$
 \int_\Omega |U_n (x)|^2 {\rm d}x =1
$$
and 
$$
\int_{\Omega_\ell} \frac{|\mathcal{E} (U_n) (x)|^2}{2} {\rm d}x + \frac{1}{2}\int_{\Omega_{n\ell}}\int_{\mathbb{R}^N} 
J(x-y)|(x-y) \cdot (U_n(y)-U_n(x))|^2\,{\rm d}y{\rm d}x
\to 0.
$$ 
Then we have that
$$
\int_{\Omega_\ell} \frac{|\mathcal{E} (U_n) (x)|^2}{2} {\rm d}x \to 0 
$$
and 
\begin{equation}
 \label{eq:omeganl-Rn-cero}
\frac{1}{2}\int_{\Omega_{n\ell}}\int_{\mathbb{R}^N} 
J(x-y)|(x-y) \cdot (U_n(y)-U_n(x))|^2\,{\rm d}y{\rm d}x
\to 0. 
\end{equation}

Since $U_n$ is bounded in $L^2(\Omega_\ell)$ (the integral in the whole $\Omega$ is 1) and 
$$
\int_{\Omega_\ell} \frac{|\mathcal{E} (U_n) (x)|^2}{2} {\rm d}x \to 0 
$$
and thanks to \eqref{eq:eqiv_H1} we get, along a subsequence, strong 
convergence of $U_n$ in $H^1(\Omega_\ell)$ to a function $K_1$ wich in turn, thanks to \eqref{eq:korn_cociente},
belongs to the space $RM$.

Now we argue in the nonlocal part $\Omega_{n\ell}$. Since $U_n$ in bounded in $L^2 (\Omega_{n\ell})$ and 
$$ 
\begin{array}{l}
\displaystyle 
\frac{1}{2}\int_{\Omega_{n\ell}}\int_{\mathbb{R}^N} J(x-y)|(x-y) \cdot (U_n(y)-U_n(x))|^2\, {\rm d}y{\rm d}x \\[7pt]
\displaystyle =
\frac{1}{2}\int_{\Omega_{n\ell}}\int_{\Omega_{n\ell}} J(x-y)|(x-y) \cdot (U_n(y)-U_n(x))|^2\, {\rm d}y{\rm d}x \\[7pt]
\displaystyle \qquad +
\frac{1}{2}\int_{\Omega_{n\ell}}\int_{\Omega_{\ell}} J(x-y)|(x-y) \cdot (U_n(y)-U_n(x))|^2\, {\rm d}y{\rm d}x
\\[7pt]
\displaystyle \qquad +
\frac{1}{2}\int_{\Omega_{n\ell}}\int_{\mathbb{R}^N \setminus \Omega} J(x-y) |(x-y) \cdot (U_n(y)-U_n(x))|^2\, {\rm d}y{\rm d}x
\to 0.
\end{array}
$$
we have that (extracting another subsequence)
$
U_n \rightharpoonup U
$
weakly in $L^2(\Omega_{n\ell},\mathbb{R}^N)$ and therefore it does
$W_n(x,y)=U_n(y)-U_n(x)$ in 
$L^2(\Omega_{n\ell}\times\Omega_{n\ell},\mathbb{R}^N)$. Weakly lower semicontinuity for the convex  functional
$$
\frac{1}{2}\int_{\Omega_{n\ell}}\int_{\Omega_{n\ell}} J(x-y)|(x-y) \cdot W(x,y)|^2\, {\rm d}y{\rm d}x, 
$$
gives 
$$
\begin{array}{l}
\displaystyle 
\frac{1}{2}\int_{\Omega_{n\ell}}\int_{\Omega_{n\ell}} J(x-y)|(x-y) \cdot (U(y)-U(x))|^2\, {\rm d}y{\rm d}x 
 \\[7pt]
 \qquad \displaystyle \leq \lim_{n \to \infty}\frac{1}{2}\int_{\Omega_{n\ell}}\int_{\Omega_{n\ell}} J(x-y)|(x-y)\cdot(U_n(y)-U_n(x))|^2\, {\rm d}y{\rm d}x =0.
 \end{array}
$$
With a similar argument, we get
$$
\begin{array}{l}
\displaystyle 
\frac{1}{2}\int_{\Omega_{n\ell}}\int_{\Omega_{\ell}} J(x-y)|(x-y) \cdot (K_1-U(x))|^2\, {\rm d}y{\rm d}x \\[7pt]
\qquad \displaystyle \leq \lim_{n\to \infty}
\frac{1}{2}\int_{\Omega_{n\ell}}\int_{\Omega_{\ell}} J(x-y)|(x-y) \cdot (U_n(y)-U_n(x))|^2\, {\rm d}y{\rm d}x = 0,
\end{array}
$$
where we have used strong convergence of $U_n$ in $\Omega_{l}$
From the first inequality and 
\eqref{eq:rm-debil} we obtain that $U \equiv K_2 \in RM$ in $\Omega_{n\ell}$. 

The second inequality yields 
$$
\frac{1}{2}\int_{\Omega_{n\ell}}\int_{\Omega_{\ell}} J(x-y)|(x-y) \cdot (K_1(y)-K_2 (x))|^2\, {\rm d}y{\rm d}x =0
$$
and then
$
(x-y) \cdot (K_1(y)-K_2 (x)) = 0, 
$
that is,
$
(x-y) \cdot (M_1y-M_2 x + b) = 0 
$
a.e. in $D^\delta=\{ (x,y)\in \Omega_{n\ell}\times \Omega_{\ell}:\|x-y\|< \delta\}.$
 Since matrices
$M_i$, for $1\le i\le 2$, are skew symmetric $w\cdot M_iw=0$ for any $w\in \RR^n$ and hence, writting $M_1y=M_1(y-x)+M_1x$ and calling $B$ the skew symmetric matrix $B=M_1-M_2$, we have  
$
(x-y) \cdot (Bx + b) = 0 
$
a.e. $(x,y)\in D^\delta$. 
 Let us take a pair  $(x,y)\in D^\delta$, and open sets such that $B_x,B_y\subset \RR^n$ $B_x\times B_y \subset D^{\delta}$, $x\in B_x, y\in B_y$. Since $B_x$ is open we see that there exists
a set $X=\{x_1,x_2,\cdots,x_n\}\subset B_x$, of linearly independent vectors.
For each $i$ we have 
$
(x_i-y) \cdot (Bx_i + b) = 0, 
$
for all $y\in B_y$ and since $B_y$ is open we can find $n$ sets  $W_i=\{x_i-y_1^i,x_i-y_2^i,\cdots,x_i-y_n^i \}$ $1\le i\le n$, of linearly independent vectors, with $y_j^i\in B_y$ and such that, for each $i$, we have  
$
(x_i-y^i_j) \cdot (Bx_i + b) = 0$, $1\le j\le n,
$
which says that $Bx_i + b=0$ for all the elements of the basis $X$. As a consequence $b=0$ and $M_1-M_2=B=0$, in particular $K_1=K_2$. Since $U_n \in H$ we have that
$U_n =0 \mbox{ in } \mathbb{R}^N \setminus \Omega $
and then we obtain that
$
K_1 = K_2 = 0.
$

Up to now we have that
$U_n \to 0 $ strongly in $H^1(\Omega_{\ell})$ and $U_n \to 0$ weakly in $L^2(\Omega_{n\ell})$.
Then, as
$$
1= \int_\Omega |U_n (x)|^2 {\rm d}x = \int_{\Omega_{\ell}} |U_n (x)|^2 {\rm d}x + \int_{\Omega_{n\ell}} |U_n (x)|^2 {\rm d}x
$$
we get that
$$
 \int_{\Omega_{n\ell}} |U_n (x)|^2 {\rm d}x \to 1.
$$

On the other hand, \eqref{eq:omeganl-Rn-cero} implies in particular that \eqref{eq:nolcalomega_vect} holds, and therefore \ref{lema:contagio_vectorial} says that 
 $$
 \lim_{n\to\infty}\int_{\Omega_{n\ell}}|U_n(x)|^2 {\rm d}x =0,
 $$
 a contradiction.
\end{proof}
\begin{proof}[Proof of \ref{teo.elast.88}] Since  
$$
E_{I}(U) \geq c \int_{\Omega}|U (x)|^2 {\rm d}x - \int_\Omega F(x)U(x) {\rm d}x, 
$$
existence of a minimizer follows, once more,
by the direct method of calculus of variations. Therefore, for an appropriate  $U\in H_I$,
$$E_I(U) = \min_{V\in H_I} E_{I}(V),$$ 
 while uniqueness is granted thanks to the  convexity of the energy functional.

Now we choose $\varPhi$ smooth enough $\varPhi\in H_I$ and compute 
$$
E_{I}(U+t\varPhi)-E_{I}(U).
$$
For the local part we notice
$$
\Xi(U+t\varPhi)-\Xi(U)=t\left(2\mu\int_{\Omega_\ell}\mathcal{E}(U):\mathcal{E}(\varPhi)+\mu\int_{\Omega_\ell}div(U)div(\varPhi)\right)+
t^2\Xi(\varPhi),
$$
where $A:B$ stands for the scalar product $A:B=tr(A^TB)$. 
Taking into account that $A:B=0$ for any symmetric matrix $A$ and any skew-symmetric matrix $B$, we  see that
$$
\begin{array}{l}
\displaystyle
2\mu\int_{\Omega_\ell}\mathcal{E}(U):\mathcal{E}(\varPhi)+\mu\int_{\Omega_\ell}div(U)div(\varPhi)
\\[7pt]
\displaystyle =\int_{\Omega_\ell}(2\mu \mathcal{E}(U)+\lambda div(U) Id):\nabla \varPhi=\int_{\Omega_\ell}\sigma (U):\nabla \varPhi.
\end{array}
$$
On the other hand, calling
$$
\Theta(U)=\frac{1}{2}\int_{\Omega_{n\ell}}\int_{\mathbb{R}^N} 
J(x-y)|(x-y) \cdot (U(y)-U(x))|^2\,{\rm d}y{\rm d}x,
$$
we have
$$
\Theta(U+t\varPhi)=\frac{1}{2}\int_{\Omega_{n\ell}}\int_{\mathbb{R}^N} 
J(x-y)\left[(x-y) \cdot (U(x)-U(y))+t(\varPhi(x)-\varPhi(y))\right]^2\,{\rm d}y{\rm d}x,
$$
hence
$$
\begin{array}{l}
\displaystyle 
\Theta(U+t\varPhi)-\Theta(U) \\[7pt]
\displaystyle =t\int_{\Omega_{n\ell}} \!\!\int_{\mathbb{R}^N} 
J(x-y)\!\!\left[\!(x-y)\otimes (x-y)\!\right] \!\! (U(x)-U(y))^T (\varPhi(x)-\varPhi(y))\,{\rm d}y{\rm d}x+t^2\Theta(\varPhi).
\end{array}
$$
Now we call
$$
I=\int_{\Omega_{n\ell}}\int_{\mathbb{R}^N} 
J(x-y)\left[(x-y)\otimes (x-y)\right]  (U(x)-U(y))^T\cdot(\varPhi(x)-\varPhi(y))\,{\rm d}y{\rm d}x=I_1+I_2,$$
with
$$
I_1=\int_{\Omega_{n\ell}}\int_{\Omega_{n\ell}} 
J(x-y)\left[(x-y)\otimes (x-y)\right]  (U(x)-U(y))^T\cdot(\varPhi(x)-\varPhi(y))\,{\rm d}y{\rm d}x$$
and 
$$
I_2=\int_{\Omega_{n\ell}}\int_{\RR^n\setminus\Omega_{n\ell}} 
J(x-y)\left[(x-y)\otimes (x-y)\right]  (U(x)-U(y))^T\cdot(\varPhi(x)-\varPhi(y))\,{\rm d}y{\rm d}x$$
for $I_1$ we get
$$
I_1=2\int_{\Omega_{n\ell}}\left(\int_{\Omega_{n\ell}} 
J(x-y)\left[(x-y)\otimes (x-y)\right]  (U(x)-U(y))^T{\rm d}y\right)\cdot\varPhi(x)\,{\rm d}x,$$
while
\begin{equation}
\begin{split} 
I_2=&\int_{\Omega_{n\ell}}\left(\int_{\RR^n\setminus\Omega_{n\ell}} 
J(x-y)\left[(x-y)\otimes (x-y)\right]  (U(x)-U(y))^T \,{\rm d}y\right)\cdot\varPhi(x){\rm d}x\\
&-\int_{
\Omega_{n\ell}}\int_{\RR^n\setminus\Omega_{n\ell}} 
J(x-y)\left[(x-y)\otimes (x-y)\right]  (U(x)-U(y))^T\cdot\varPhi(y)\,{\rm d}y{\rm d}x\\
&=I_{21}+I_{22},\nonumber 
\end{split}
\end{equation}
since $\varPhi\equiv0$ in $\RR^n\setminus \Omega$, applying  Fubini's theorem,
$$
I_{22}=-\int_{
\Omega_{\ell}}
\left(\int_{
\Omega_{n\ell}}
J(x-y)\left[(x-y)\otimes (x-y)\right]  (U(x)-U(y))^T\,{\rm d}x\right)\cdot\varPhi(y){\rm d}y.
$$
The theorem follows.
\end{proof}

\subsection{Second model. Coupling local/nonlocal elasticity models via flux terms}
Now, we consider the local/nonlocal energy 
\begin{equation}\label{def:funcional_Dirichlet.intro.Vec_II.678}
\begin{array}{l}
\displaystyle 
E_{II}(u):=
\Xi (U) +
\frac{1}{2}\int_{\Omega_{n\ell}}\int_{\mathbb{R}^N \setminus \Omega_{\ell}} 
J(x-y)|(x-y) \cdot (U(y)-U(x))|^2\,{\rm d}y{\rm d}x \\[7pt]
\displaystyle 
\qquad \qquad + \frac{1}{2}\int_{\Omega_{n\ell}}\int_{\Gamma} 
	G(x,y)|(x-y) \cdot (U(y)-U(x))|^2\, {\rm d} \sigma(y){\rm d}x	- \int_\Omega F (x)U (x) {\rm d}x,
	\end{array}
\end{equation}
and look for minimizers of this energy in the space
$$
H_{II} = \Big\{U: U|_\Omega \in L^2 (\Omega:\mathbb{R}^N), U|_{\Omega_\ell} \in H^1 (\Omega_\ell:\mathbb{R}^N), U|_{\partial \Omega \cap \partial \Omega_\ell}=0,  
U = 0 \mbox{ in } \mathbb{R}^N \setminus \Omega \Big\}.
$$

We start by proving a modified version of \ref{lema:contagio_vectorial}.

\begin{lemma}
	\label{lema:contagio_vectorial_2}
	Let $U_n:\Omega\to \RR^n$ be a sequence such that $U_n\to 0$ strongly in $H^1(\Omega_{\ell})$ and weakly in $L^2(\Omega_{n\ell})$, if in addition  
	\begin{equation}
		\label{eq:nolcalomega_vect_2}
		\lim_{n\to\infty}\int_{\Omega_{n\ell}}\int_{\Omega_{n\ell}} J(x-y)((x-y)\cdot (U_n(x)-U_n(y)))^2 {\rm d}y {\rm d}x=0 
	\end{equation}
	and 
	\begin{equation}
		\label{eq:nonlocal_gamma_vect_2}
		\lim_{n\to\infty}\int_{\Omega_{n\ell}}\int_{\Gamma} G(x,y)((x-y)\cdot (U_n(x)-U_n(y)))^2 \, {\rm d}\sigma(y) {\rm d}x =0 
	\end{equation}
	then
	$$
	\lim_{n\to\infty}\int_{\Omega_{n\ell}}|U_n(x)|^2 {\rm d}x=0
	$$
\end{lemma}
\begin{proof}
	 As before, we sketch the proof following the lines of \ref{lema:contagio_escalar}, paying attention only to the main differences.
	From \eqref{eq:nonlocal_gamma_vect_2},   the convergence of $\{u_n\}$ and property (P),  we  find the following substitute of \eqref{eq:solonl}
	\begin{equation}
		\label{eq:solonl_vect_2}
		\int_{\Omega_{n\ell}}\int_{\Gamma} G(x,y)|(x-y) \cdot U_n(x)|^2\, {\rm d}\sigma(y){\rm d}x
		\to 0.
	\end{equation}
	Let us define,  
	$$
	A^0_\delta=\Big\{x\in \Omega_{n\ell}: dist(x,\Gamma)<\delta \Big\}.
	$$
	and arguing as in \ref{lema:contagio_escalar}, we see that
	the continuous function $g(x):A^0_{\delta}\to \RR$,  $g(x)=\sigma(B_{2\delta}(x)\cap \Gamma)$ verifies $g(x)\ge m>0$ for any  $x\in \overline{A^0_{\delta}}$. 
	As a consequence we get
	$$
	\begin{array}{l}
	\displaystyle
	\int_{\Omega_{n\ell}}\int_{\Gamma} G(x,y)\left((x-y)\cdot U_n(x)\right)^2\, {\rm d} \sigma(y) {\rm d}x \\[7pt]
	\qquad \displaystyle 
	\ge 
	\int_{A^0_\delta}\int_{B_{2\delta}(x)\cap \Gamma} G(x,y)\left((x-y)\cdot U_n(x)\right)^2\, {\rm d} \sigma(y) {\rm d}x \\[7pt]
	\qquad \displaystyle 
	\ge 
	C\int_{A^0_\delta}\int_{B_{2\delta}(x)\cap \Gamma} \left((x-y)\cdot U_n(x)\right)^2\, {\rm d} \sigma(y) {\rm d}x,
	\end{array}
	$$
	where we have also used $(G1)$.
	
	Now, consider $C_x$ the smaller  cone  at $x$ containing
	$B_{2\delta}(x)\cap \Gamma = \Gamma_x$. We would like to see that its opening angle $\alpha = \alpha(x) >0$ almost everywhere in $A^0_\delta$. Indeed, suppose $\alpha(x) = 0$. Then $B_{2\delta}(x)\cap \Gamma$ must be contained in a line. Let $r>0$ such that $\sigma( B_{2\delta}(z) \cap \Gamma_x ) > 0$ for all $z \in B_{r}(x)$. We can easily verify that $\alpha(z)>0$ almost everywhere in $B_{r}(x)$. Since this can be done for every $x \in A^0_\delta$ with $\alpha(x)=0$, we can conclude that $\alpha(x)>0$ almost everywhere in $A^0_\delta$.    
	
	Let us call  $C_u$, the cone at $x$ with axis $U_n(x)$ and opening $\frac{\pi}{2}-\alpha(x)/4$.
	Clearly $C_x\cap C_u$ contains a cone $C_I$ with an opening angle $\beta(x) \ge \alpha(x)/2$ and therefore $\sigma(C_I\cap \Gamma_x) > 0$ a.e. in $A^0_\delta$.   
	
	On the other hand, for any element $z\in C_I\subset C_u$, $ang(z,U_n(x))\le \frac{\pi}{2} -\frac{\alpha(x)}{4}$. Taking into account that for any $y\in \Gamma_x$, $y-x\in C_x$ we see that
	$$
	\begin{array}{l}
	\displaystyle 
	\int_{\Gamma_x} \left((y-x)\cdot U_n(x)\right)^2\, {\rm d}\sigma(y)   \ge 
	\int_{C_I\cap \Gamma_x} \left((y-x)\cdot U_n(x)\right)^2 {\rm d}y
	 \\[7pt]
	\qquad \displaystyle  
	\ge \sigma(C_I\cap \Gamma_x) d^2(x,\Gamma_x)\cos^2\left(\frac{\pi}{2}-\frac{\alpha(x)}{4}\right)\|U_n(x)\|^2 
	\end{array}
	$$
	since $$|(y-x)\cdot U_n(x)|\ge d(x,\Gamma_x) \cos\left(\frac{\pi}{2}-\frac{\alpha(x)}{4}\right) \|U_n(x)\|.$$ 
	Letting 
	$$g(x):= \sigma(C_I\cap \Gamma_x) d^2(x,\Gamma_x)\cos^2\left(\frac{\pi}{2}-\frac{\alpha(x)}{4}\right),$$ 
	we have $g(x)>0$ a.e. in $A^0_{\delta}$. Then, there exist some positive constant $\varepsilon$ such that $|\{g>\varepsilon\} \cap A^0_\delta|>0$. Denoting $A^0_{\delta,\varepsilon} := \{g>\varepsilon\} \cap A^0_\delta$ we have
	$$
	C 
	\int_{A^0_{\delta,\varepsilon}}\int_{C_I\cap \Gamma_x} \left((y-x)\cdot U_n(x)\right)^2 {\rm d}\sigma(y) {\rm d}x \ge 
	\varepsilon\int_{A^0_{\delta,\varepsilon}}\|U_n(x)\|^2 {\rm d}x
	$$
	showing, thanks to \eqref{eq:solonl_vect}, that $U_n\to 0$ strongly in $L^2(A^0_{\delta,\varepsilon})$. Using similar arguments, the proof can be concluded following the  steps developed in \ref{lema:contagio_escalar}.
\end{proof}


Now we are able to prove the key inequality needed to show that $E_{II}$ is coercive. 

\begin{lemma} \label{lema.1.99_II}
	There exists a positive constant $C$ such that
	$$
	\begin{array}{l}
	\displaystyle
	\Xi(U) + \frac{1}{2}\int_{\Omega_{n\ell}}\int_{\mathbb{R}^N \setminus \Omega_{\ell}} 
	J(x-y)|(x-y) \cdot (U(y)-U(x))|^2\,{\rm d}y{\rm d}x \\[7pt]
	\qquad \qquad \displaystyle 
	+ \frac{1}{2}\int_{\Omega_{n\ell}}\int_{\Gamma} 
	G(x,y)|(x-y) \cdot (U(y)-U(x))|^2\,{{\rm d}} \sigma(y){\rm d}x
	\geq C \int_\Omega |U (x)|^2 {\rm d}x,
	\end{array}
	$$
	for every $U$ in
	$$
	H_{II} = \Big\{ U \in L^2 (\Omega), U|_{\Omega_\ell} \in H^1 (\Omega_\ell), U|_{\partial \Omega \cap \partial \Omega_\ell}=0|_{\partial \Omega \cap \partial \Omega_\ell},  U =0 \mbox{ in } \mathbb{R}^N \setminus \Omega \Big\}.
	$$
\end{lemma}

\begin{proof} As before, we use that
	$$
	\Xi(U)\ge \int_{\Omega_\ell} \frac{|\mathcal{E} (U) (x)|^2}{2} {\rm d}x,
	$$
	and argue by contradiction. Assume that there is a sequence $U_n\in H$ such that
	$$
	\int_\Omega |U_n (x)|^2 {\rm d}x =1
	$$
	and 
	$$
	\int_{\Omega_\ell} \frac{|\mathcal{E} (U_n)|^2}{2} + \frac{1}{2}\int_{\Omega_{n\ell}}\int_{\mathbb{R}^N \setminus \Omega_{\ell}} 
	J(x-y)|(x-y) \cdot (U(y)-U(x))|^2\,{\rm d}y {\rm d}x
	$$
	$$
	+ \frac{1}{2}\int_{\Omega_{n\ell}}\int_{\Gamma} 
	G(x-y)|(x-y) \cdot (U(y)-U(x))|^2\, {{\rm d}} \sigma (y){\rm d}x
	\to 0.
	$$ 
	Then, we have that
	$$
	\int_{\Omega_\ell} \frac{|\mathcal{E} (U_n) (x)|^2}{2} {\rm d}x \to 0, 
	$$ 
	\begin{equation}
		\label{eq:omeganl-Rn-cero_II}
		\int_{\Omega_{n\ell}}\int_{\mathbb{R}^N \setminus \Omega_{\ell}} 
		J(x-y)|(x-y) \cdot (U(y)-U(x))|^2\,{\rm d} y {\rm d}x \to 0,
	\end{equation}
	and
	\begin{equation}
		\label{eq:omeganl-gamma-cero_II}
		\int_{\Omega_{n\ell}}\int_{\Gamma} 
		G(x-y)|(x-y) \cdot (U(y)-U(x))|^2\,{{\rm d}}\sigma(y){\rm d}x\to 0.
	\end{equation}

	Since $U_n$ is bounded in $L^2(\Omega_\ell)$ (the integral in the whole $\Omega$ is 1) and 
	$$
	\int_{\Omega_\ell} \frac{|\mathcal{E} (U_n) (x)|^2}{2} {\rm d}x \to 0 
	$$
	and thanks to \eqref{eq:eqiv_H1} we get, along a subsequence, strong 
	convergence of $U_n$ in $H^1(\Omega_\ell)$ to a function $K_1$ wich in turn, thanks to \eqref{eq:korn_cociente},
	belongs to the space $RM$.

	Now we argue in the nonlocal part $\Omega_{n\ell}$. Since $U_n$ in bounded in $L^2 (\Omega_{n\ell})$ and 
	$$ 
	\begin{array}{l}
		\displaystyle 
		\frac{1}{2}\int_{\Omega_{n\ell}}\int_{\mathbb{R}^N \setminus \Omega_{\ell}} J(x-y)|(x-y) \cdot (U_n(y)-U_n(x))|^2\, {\rm d}y{\rm d}x \\[7pt]
		\displaystyle =
		\frac{1}{2}\int_{\Omega_{n\ell}}\int_{\Omega_{n\ell}} J(x-y)|(x-y) \cdot (U_n(y)-U_n(x))|^2\, {\rm d}y{\rm d}x \\[7pt]
		\displaystyle
		\displaystyle \qquad +
		\frac{1}{2}\int_{\Omega_{n\ell}}\int_{\mathbb{R}^N \setminus \Omega} J(x-y) |(x-y) \cdot (U_n(y)-U_n(x))|^2\, {\rm d}y{\rm d}x
		\to 0
	\end{array}
	$$
	we have that (extracting another subsequence)
	$
	U_n \rightharpoonup U
	$
	weakly in $L^2(\Omega_{n\ell},\mathbb{R}^N)$ and therefore it also converges
	$W_n(x,y)=U_n(y)-U_n(x)$ in 
	$L^2(\Omega_{n\ell}\times\Omega_{n\ell},\mathbb{R}^N)$. Weakly lower semicontinuity (due to convexity) of the functional
	$$
	\frac{1}{2}\int_{\Omega_{n\ell}}\int_{\Omega_{n\ell}} J(x-y)|(x-y) \cdot W(x,y)|^2\, {\rm d}y{\rm d}x, 
	$$
	gives 
	$$
	\begin{array}{l}
	\displaystyle 
	\frac{1}{2}\int_{\Omega_{n\ell}}\int_{\Omega_{n\ell}} J(x-y)|(x-y) \cdot (U(y)-U(x))|^2\, {\rm d}y{\rm d}x 
	\\[7pt]
	\qquad \displaystyle 
	\leq \lim_{n \to \infty}\frac{1}{2}\int_{\Omega_{n\ell}}\int_{\Omega_{n\ell}} J(x-y)|(x-y)\cdot(U_n(y)-U_n(x))|^2\, {\rm d}y{\rm d}x =0.
	\end{array}
	$$
	With a similar argument and using \eqref{eq:omeganl-gamma-cero_II} we obtain
	$$
	\begin{array}{l}
	\displaystyle 
	\frac{1}{2}\int_{\Omega_{n\ell}}\int_{\Gamma} G(x,y)|(x-y) \cdot (K_1-U(x))|^2\, {{\rm d}}\sigma(y){\rm d}x  \\[7pt]
	\qquad \displaystyle \leq \lim_{n\to \infty}
	\frac{1}{2}\int_{\Omega_{n\ell}}\int_{\Gamma} G(x,y)|(x-y) \cdot (U_n(y)-U_n(x))|^2\, {{\rm d}}\sigma(y){\rm d}x = 0,
	\end{array}
	$$
	where we have used strong convergence of $U_n$ in $L^2(\Gamma$), due to the strong convergence of $U_n$ in $H^1(\Omega_{\ell})$.
	From the first inequality and 
	\eqref{eq:rm-debil} we obtain that $U \equiv K_2 \in RM$ in $\Omega_{n\ell}$. 
	Now, the second inequality yields 
	$$
	\frac{1}{2}\int_{\Omega_{n\ell}}\int_{\Gamma} J(x-y)|(x-y) \cdot (K_1(y)-K_2 (x))|^2\, {{\rm d}} \sigma(y){\rm d}x =0,
	$$
	and hence
	$
	(x-y) \cdot (K_1(y)-K_2 (x)) = 0,
	$
	that is,
	$
	(x-y) \cdot (M_1y-M_2 x + b) = 0 
	$
	a.e. in $D^\delta=\{ (x,y)\in \Omega_{n\ell}\times \Gamma :\|x-y\|< \delta\}.$
	Since the matrices
	$M_i$, for $1\le i\le 2$, are skew symmetric we have $w\cdot M_iw=0$ for any $w\in \RR^n$ and hence, writing $M_1y=M_1(y-x)+M_1x$ and calling $B$ the skew symmetric matrix $B=M_1-M_2$, we have  
	$
	(x-y) \cdot (Bx + b) = 0 
	$
	a.e. $(x,y)\in D^\delta$. 
	Let us take a pair  $(x,y)\in D^\delta$, and open sets such that $B_x,B_y\subset \RR^n$ $B_x\times B_y\cap\Gamma \subset D^{\delta}$, $x\in B_x, y\in B_y \cap \Gamma$ with $\sigma(B_y \cap \Gamma)>0$. Since $\sigma(B_y \cap \Gamma)>0$ there exist a set $Y = \{y_1,...,y_N\} \subset B_y \cap \Gamma$ with $y_i \not = y_j$, $1 \leq i , j \leq N$. Now, since $B_x$ is an open set, we can choose a set of linearly independent vectors $X = \{x_1,...,x_N\} \in B_x$ in such a way that $W_i = \{x_i-y_1,x_i - y_2,...,x_i-y_N\}$, $y_j \in Y$, is also a linearly independent set for every $1\leq i \leq N$.    
	and such that, for each $i$, we have  
	$
	(x_i-y_j) \cdot (Bx_i + b) = 0$, $1\le j\le N.
	$
	which says that $Bx_i + b=0$ for all the elements of the basis $X$. As a consequence $b=0$ and $M_1-M_2=B=0$, in particular $K_1=K_2$. Since $U_n \in H$ we have that
	$U_n =0 \mbox{ in } \mathbb{R}^N \setminus \Omega 
	$
	and then we obtain that
	$
	K_1 = K_2 = 0.
	$
	
	Up to now we have that
	$U_n \to 0 $ strongly in $H^1(\Omega_{\ell})$ and $L^2(\Gamma)$, and $U_n \to 0$ weakly in $L^2(\Omega_{n\ell})$.
	Then, as
	$$
	1= \int_\Omega |U_n(x)|^2 {\rm d}x= \int_{\Omega_{\ell}} |U_n (x)|^2 {\rm d}x + \int_{\Omega_{n\ell}} |U_n (x)|^2 {\rm d}x
	$$
	we get that
	$$
	\int_{\Omega_{n\ell}} |U_n (x)|^2 {\rm d}x \to 1.
	$$
	
	On the other hand,  \eqref{eq:omeganl-Rn-cero_II} and \eqref{eq:omeganl-gamma-cero_II} say in particular that \eqref{eq:nolcalomega_vect_2} and \eqref{eq:nonlocal_gamma_vect_2} hold, and therefore \ref{lema:contagio_vectorial_2} says that 
	$$
	\lim_{n\to\infty}\int_{\Omega_{n\ell}}|U_n(x)|^2 {\rm d}x =0,
	$$
	a contradiction.
\end{proof}
\begin{proof}[{Proof of \ref{teo.elast.88-999}}]
Existence of a unique minimizer follows as in the first model case. Finally, we derive a local-nonlocal equation as before.  
Choose $\varPhi$ smooth enough $\varPhi\in H_{II}$ and compute 
$$
E_{II}(U+t\varPhi)-E_{II}(U).
$$
For the local part we notice
$$
\Xi(U+t\varPhi)-\Xi(U)=t\left(2\mu\int_{\Omega_\ell}\mathcal{E}(U):\mathcal{E}(\varPhi)+\mu\int_{\Omega_\ell}div(U)div(\varPhi)\right)+
t^2\Xi(\varPhi),
$$
where $A:B$ stands for the scalar product $A:B=tr(A^TB)$. 
Taking into account that $A:B=0$ for any symmetric matrix $A$ and any skew-symmetric matrix $B$, we  see that
$$
\begin{array}{l}
\displaystyle 
2\mu\int_{\Omega_\ell}\mathcal{E}(U):\mathcal{E}(\varPhi)+\mu\int_{\Omega_\ell}div(U)div(\varPhi)
\\[7pt] \displaystyle =\int_{\Omega_\ell}(2\mu \mathcal{E}(U)+\lambda div(U) Id):\nabla \varPhi=\int_{\Omega_\ell}\sigma (U):\nabla \varPhi.
\end{array}
$$
On the other hand, calling

$$
\Theta(U)=\frac{1}{2}\int_{\Omega_{n\ell}}\int_{\mathbb{R}^N \setminus \Omega_{\ell}} 
J(x-y)|(x-y) \cdot (U(y)-U(x))|^2\,{\rm d}y{\rm d}x,
$$
we have
$$
\Theta(U+t\varPhi)=\frac{1}{2}\int_{\Omega_{n\ell}}\int_{\mathbb{R}^N \setminus \Omega_{\ell}} 
J(x-y)\left[(x-y) \cdot (U(x)-U(y))+t(\varPhi(x)-\varPhi(y))\right]^2\,{\rm d}y{\rm d}x,
$$
hence
$$
\begin{array}{l}
\displaystyle 
\Theta(U+t\varPhi)-\Theta(U) \\[7pt]
\displaystyle =t\int_{\Omega_{n\ell}} \!\!\int_{\mathbb{R}^N \setminus \Omega_{\ell}} 
\!\!\!\!J(x-y)\!\! \left[\!(x-y)\otimes (x-y)\!\right] \!\! (\!U(x)-U(y)\!)^T\!(\!\varPhi(x)-\varPhi(y)\!){\rm d}y{\rm d}x+t^2\Theta(\varPhi).
\end{array}
$$
Now we call
$$
I=\int_{\Omega_{n\ell}}\!\!\int_{\mathbb{R}^N\setminus \Omega_{\ell}} 
\!\!\!\!J(x-y)\!\!\left[\!(x-y)\otimes (x-y) \!\right] \!\! (\!U(x)-U(y)\!)^T\!\!(\varPhi(x)-\varPhi(y))\,{\rm d}y{\rm d}x=I_1+I_2,$$
with
$$
I_1=\int_{\Omega_{n\ell}}\int_{\Omega_{n\ell}} J(x-y)\left[(x-y)\otimes (x-y)\right]  (U(x)-U(y))^T\cdot(\varPhi(x)-\varPhi(y))\,{\rm d}y{\rm d}x$$
and 
$$
I_2=\int_{\Omega_{n\ell}}\int_{\RR^n\setminus\Omega} 
J(x-y)\left[(x-y)\otimes (x-y)\right]  (U(x)-U(y))^T\cdot(\varPhi(x)-\varPhi(y))\,{\rm d}y{\rm d}x$$
for $I_1$ we get
$$
I_1=2\int_{\Omega_{n\ell}}\left(\int_{\Omega_{n\ell}} 
J(x-y)\left[(x-y)\otimes (x-y)\right]  (U(x)-U(y))^T{\rm d}y\right)\cdot\varPhi(x)\,{\rm d}x,$$
while
\begin{equation}
\begin{split} 
I_2=&\int_{\Omega_{n\ell}}\left(\int_{\RR^n\setminus\Omega} 
J(x-y)\left[(x-y)\otimes (x-y)\right]  (U(x)-U(y))^T \,{\rm d}y\right)\cdot\varPhi(x){\rm d}x\\
&-\int_{
\Omega_{n\ell}}\int_{\RR^n\setminus\Omega} 
J(x-y)\left[(x-y)\otimes (x-y)\right]  (U(x)-U(y))^T\cdot\varPhi(y)\,{\rm d}y{\rm d}x\\
&=I_{21}+I_{22},\nonumber 
\end{split}
\end{equation}
since $\varPhi\equiv0$ in $\RR^n\setminus \Omega$, applying we have $I_{22} = 0$.
Finally, we have the term 
$$
\Theta_{\Gamma}(U) = \frac{1}{2}\int_{\Omega_{n\ell}}\int_{\Gamma} 
G(x,y)|(x-y) \cdot (U(y)-U(x))|^2\,{\rm d}\sigma(y){\rm d}x, 
$$
and, arguing as before, we obtain
$$
\begin{array}{l}
\displaystyle 
\Theta_{\Gamma}(U+t\varPhi)-\Theta_{\Gamma}(U) \\[7pt]
\displaystyle =t\int_{\Omega_{n\ell}} \!\! \int_{\Gamma} 
\!\! G(x,y)\!\left[\!(x-y)\otimes (x-y)\!\right] \! (\!U(x)-U(y)\!)^T\!\!(\!\varPhi(x)-\varPhi(y)\!){\rm d}\sigma(y){\rm d}x+t^2\Theta_{\Gamma}(\varPhi).
\end{array}
$$
Proceeding as we did for the second scalar model, we have 
$$
\int_{\Omega_{n\ell}}\int_{\Gamma} 
G(x,y)\left[(x-y)\otimes (x-y)\right]  (U(x)-U(y))^T\cdot(\varPhi(x)-\varPhi(y))\,{\rm d}\sigma(y){\rm d}x
$$
$$
= \int_{\Omega_{n\ell}}\int_{\Gamma} 
G(x,y)\left[(x-y)\otimes (x-y)\right]  (U(x)-U(y))^T\,{\rm d}\sigma(y)\cdot\varPhi(x){\rm d}x
$$
$$
+ \int_{\Gamma} \int_{\Omega_{n\ell}}
G(y,x)\left[(x-y)\otimes (x-y)\right]  (U(x)-U(y))^T{\rm d}y\cdot\varPhi(x)\,{\rm d}\sigma(x)
$$
and the lemma follows.
\end{proof}

\section{Possible extensions of our results} 
\label{sect-comments}

Finally, let us comment briefly on
possible extensions of our results.

\subsection{Inhomogeneous equations} Pure local or nonlocal models are well suited for homogeneous environments. 
When one deals with an inhomogeneous 
diffusion process one possibility is to add a diffusion coefficient and consider operators like
\begin{equation} \label{op.1}
 \mbox{div} \Big( a(x)\nabla u \Big) (x),
 \qquad \mbox{or} \qquad 
\int_{\mathbb{R}^N} b(x,y)J(x-y)(u(y)-u(x))
\, {\rm d}y.
\end{equation}
To obtain inhomogeneous equations one is lead to consider terms like
$$
\int_{\Omega_\ell} \frac{a(x) |\nabla u(x)|^2}{2} dx,
\qquad
\frac{1}{2}\int_{\Omega_{n\ell}}\int_{\Omega_{n\ell}} b(x,y) J(x-y)(u(y)-u(x))^2\, {\rm d}y{\rm d}x
$$
and 
$$
\frac{1}{2}\int_{\Omega_{n\ell}}\int_{\Omega_{\ell}} c(x,y) J(x-y)(u(y)-u(x))^2\, {\rm d}y{\rm d}x
$$
in our energies.

As long as the coefficients $a(x)$, $b(x,y)$ and $c(x,y)$ are positive and bounded (both from adobe and from below away from zero) the same computations
that we made here work and we have existence and uniqueness of a minimizer that verifies an equation in which
the operators given by \eqref{op.1} appear.

\subsection{Singular kernels} With the same ideas used here we can deal with 
singular kernels. For example, one can show that there is a unique minimizer of
\begin{equation}\label{def:funcional_Dirichlet.intro.singular}
E_i(u):=\int_{\Omega_\ell} \frac{|\nabla u(x)|^2}{2} {\rm d}x+ \frac{1}{2}\int_{\Omega_{n\ell}}\int_{\mathbb{R}^N} \frac{C}{|x-y|^{N+2s}}
(u(y)-u(x))^2\,{\rm d}y{\rm d}x
- \int_\Omega f (x) u (x) {\rm d}x
\end{equation}
in
$$
H = \Big\{ u \in H^s (\mathbb{R}^N), u|_{\Omega_\ell} \in H^1 (\Omega_\ell), 
u = 0 \mbox{ in } \mathbb{R}^N \setminus \Omega \Big\}.
$$
In this case in the nonlocal region a fractional Laplacian 
$$
\Delta^s u(x) = \int_{\mathbb{R}^N} \frac{C}{|x-y|^{N+2s}}
(u(y)-u(x))\,{\rm d}y 
$$
appears. 
Concerning mixed couplings we can also consider an energy of the form 
\begin{equation}\label{def:funcional_dirichlet.intro.2257.99}
\begin{array}{l}
\displaystyle
E_{ii}(u):=\int_{\Omega_\ell} \frac{|\nabla u (x)|^2}{2} {\rm d}x 
+ \frac{1}{2}\int_{\Omega_{n\ell}}\int_{\mathbb{R}^N\setminus \Omega_{\ell} }
\frac{C}{|x-y|^{N+2s}} (u(y)-u(x))^2\, {\rm d}y{\rm d}x \\[7pt]
\qquad \qquad \quad \displaystyle 
+ \frac{1}{2}\int_{\Omega_{n\ell}}\int_{\Gamma}G(x,z)\left(u(x)-u(z)\right)^2 {\rm d}\sigma(z) dx
- \int_\Omega f (x) u(x) {\rm d}x
\end{array}
\end{equation}
in the same $H$ as before.

Now, since in fractional Sobolev spaces we have a trace theorem, 
we can go one step further and consider couplings between the local and the nonlocal parts
in sets of smaller dimension (both for the local and the nonlocal parts). 
Fix two subsets 
$
\Gamma_\ell \subset  \partial \Omega_\ell$,  and $ \Gamma_{n\ell} \subset  \partial \Omega_{n\ell}
$ 
and consider an energy of the form 
\begin{equation}\label{def:funcional_dirichlet.intro.2257.99.44}
\begin{array}{l}
\displaystyle
\widetilde{E}_{ii}(u):=\int_{\Omega_\ell} \frac{|\nabla u (x)|^2}{2} {\rm d}x 
+ \frac{1}{2}\int_{\Omega_{n\ell}}\int_{\mathbb{R}^N\setminus \Omega_{\ell} }
\frac{C}{|x-y|^{N+2s}} (u(y)-u(x))^2\, {\rm d}y{\rm d}x \\[7pt]
\qquad \qquad \quad \displaystyle 
+ \frac{1}{2}\int_{\Gamma_{n\ell}}\int_{\Gamma_\ell}G(x,z)\left(u(x)-u(z)\right)^2 {\rm d}\sigma(z) {\rm d} \sigma(x)
- \int_\Omega f (x) u(x) {\rm d}x.
\end{array}
\end{equation}
For $s>\frac12$ we have compactness of the Sobolev trace embedding from $H^1(\Omega_\ell )$ into $L^2(\Gamma_\ell)$ and
from  $H^s(\Omega_{n\ell} )$ into $L^2(\Gamma_{n\ell})$. Therefore, we have a well defined energy functional in  
$$
H = \Big\{ u \in H^s (\mathbb{R}^N), u|_{\Omega_\ell} \in H^1 (\Omega_\ell), 
  u = 0 \mbox{ in } \mathbb{R}^N \setminus \Omega \Big\}.
$$

\subsection{Nonlinear problems} We can also tackle nonlinear problems associated to functionals like
\begin{equation}\label{def:funcional_neumann.intro.p.r.66}
E^{p,r}(u):=\int_{\Omega_\ell} \frac{|\nabla u (x)|^p}{p} {\rm d}x + \frac{1}{r}\int_{\Omega_{n\ell}}\int_{\Omega} J(x-y)(u(y)-u(x))^r\,{\rm d}y{\rm d}x
- \int_\Omega f(x) u(x) {\rm d}x.
\end{equation}
In this case the natural space to look for minimizers is 
in
$$
H = \Big\{ u \in L^r (\Omega), u|_{\Omega_\ell} \in W^{1,p} (\Omega_\ell), 
  u = 0 \mbox{ in } \mathbb{R}^N \setminus \Omega \Big\}.
$$
Now, we obtain a $p-$Laplacian operator in the local part,
$$
\Delta_p u(x) = \mbox{div}(|\nabla u|^{p-2} \nabla u) (x)
$$
and a nonlocal $r-$Laplacian in the nonlocal part
$$
L (u) (x) = \int_{\mathbb{R}^N} J(x-y) |u(y)-u(x)|^{r-2}
(u(y)-u(x))\,{\rm d}y.
$$

\subsection{More general vectorial models}
\label{sec:caso_mas_gral}
The key inequality 
$$
E(U) \geq c \int_{\Omega}|U (x)|^2 {\rm d}x - \int_\Omega F(x)U(x) {\rm d}x, 
$$
with $E(U)=E_I(U),E_{II}(U)$, can be easily obtained for more general local energies $\Xi(U)$  by following the same arguments used along the vectorial section. This is indeed the case, for instance, under the assumption
$$
C\Xi(U) + \int_{\Omega} U^2(x)\, {\rm d}x\ge \int_{\Omega} \nabla U(x)^2\, {\rm d}x,
$$
if in addition we have 
lower semicontinuity of $\Xi$ in $H^1(\Omega_{\ell})^{N}$ 
together with the condition $ker(\Xi)\subseteq RM$ we can obtain existence of minimizers of the corresponding energy. 
The obvious choice 
$$
\Xi(U)=\int_{\Omega} W(\nabla U(x))\,dx
$$
with 
$
W:\RR^{N\times N}\to \RR,
$
convex, bounded by below and with a growth condition of the form $W(M)\ge C\|M\|^2$, verifies these assumptions.

Nonetheless, it is well known that convex stored energy funtions are not appropriate for general elastic materials, partially due  to the restriction $det(\nabla U)>0$  (needed in order to prevent interpenetration of matter).  This fact, together with the expected non-uniqueness of minimizers in certain contexts (e.g. buckling) has called for surrogates of convexity, such as polyconvexity, cuasiconvexity,  rank-one convexity, etc. We do not treat here couplings  involving these 
kinds of energies.

\medskip

\medskip


\bibliographystyle{plain}

\begin{thebibliography}{99}




\bibitem{Peri1} Azdoud, Y.; Han, F.; Lubineau, G. {\it A morphing framework to couple non-local and local anisotropic
continua}. Inter. J. Solids Structures 50(9),  (2013), 1332--1341.

\bibitem{AD}
  Acosta, G.;  Dur\'an R.
\newblock {Divergence Operator and Related Inequalities}.
\newblock SpringerBriefs in Mathematics, 2017.

\bibitem{ElLibro}
  Andreu-Vaillo, F.;  Toledo-Melero, J.; Mazon, J. M.;  
  Rossi, J. D.
\newblock {Nonlocal diffusion problems}.
\newblock Number 165. American Mathematical Soc., 2010.

\bibitem{Peri2} Badia, S.; Bochev, P.; Lehoucq, R.; Parks, M.; Fish, J.; Nuggehally, M.A.; Gunzburger, M. {\it A forcebased
blending model for atomistic-to-continuum coupling}. Inter. J. Multiscale
Comput. Engineering, 5(5), (2007), 387--406.

\bibitem{Peri3} Badia, S.; Parks, M.; Bochev, P.; Gunzburger, M.; Lehoucq, R. {\it On atomistic-to-continuum coupling
by blending}. Multiscale Modeling Simulation, 7(1), (2008), 381--406.

\bibitem{BCh} Bates, P.; Chmaj, A. \emph{An integrodifferential
model for phase transitions: stationary solutions in higher
dimensions}. J. Statist. Phys. 95 (1999), no.\,5--6, 1119--1139.

\bibitem{Bere} Berestycki, H.; Coulon, A.-Ch.; Roquejoffre, J-M.; Rossi, L. 
{ \it The effect of a line with nonlocal diffusion on Fisher-KPP propagation.} Math. Models Meth.  Appl. 
Sciences, 25.13, (2015), 2519--2562.


\bibitem{brezis}
Brezis, H. \emph{Functional analysis, Sobolev spaces and partial differential equations}. Springer Science \& Business Media, 2010.


\bibitem{CaRo} Capanna, M.; Rossi, J. D. {\it Mixing local and nonlocal evolution equations.} Preprint.

\bibitem{CF} Carrillo, C.; Fife, P. \emph{Spatial effects in discrete generation population models}. J. Math. Biol. 50 (2005), no.\,2, 161--188.


\bibitem{ChChRo} Chasseigne, E.; Chaves, M.; Rossi, J.\,D. \emph{Asymptotic behavior for nonlocal diffusion equations}. J. Math. Pures Appl. (9) 86 (2006), no.\,3, 271--291.

\bibitem{Cia} Ciarlet, P. G. Three-Dimensional Elasticity.
Mathematical Elasticity, Vol. 1,
Studies in mathematics and its applications, 1994.

\bibitem{CERW} Cort\'azar, C.; M. Elgueta, M.; Rossi, J.\,D.; Wolanski, N. \emph{Boundary fluxes for non-local diffusion}.   J. Differential Equations 234 (2007), no.\,2, 360--390.


\bibitem{Cortazar-Elgueta-Rossi-Wolanski}  Cort\'azar, C.; M. Elgueta, M.; Rossi, J.\,D.; Wolanski, N. \emph{How to approximate the heat equation with Neumann boundary conditions by nonlocal diffusion problems}. Arch. Ration. Mech. Anal. 187 (2008), no.\,1, 137--156.

\bibitem{Coville} Coville, J.; Dupaigne, L. {\it On a non-local equation arising in population dynamics}. Proc. Royal Soc. Edinburgh A: Mathematics, 
(2007), 137.4: 727--755.


\bibitem{delia2} D'Elia, M.; Perego, M.; Bochev, P.; Littlewood, D. \emph{A coupling strategy for nonlocal and local diffusion models with mixed volume constraints and boundary conditions}. Comput. Math. Appl. 71 (2016), no.\,11, 2218--2230.

\bibitem{delia3} D'Elia, M.; Ridzal, D.; Peterson, K.\,J.; Bochev, P.; Shashkov, M. \emph{Optimization-based mesh correction with volume and convexity constraints}. J. Comput. Phys. 313 (2016), 455--477.


\bibitem{delia}
D'Elia, M.; Bochev, P. \emph{Formulation, analysis and computation of an optimization-based local-to-nonlocal coupling method}. arXiv:1910.11214.

\bibitem{SUR} 
D'Elia, M.; Li, X.; Seleson, P.; Tian, X.; Yu, Y. {\it A review of Local-to-Nonlocal coupling methods in nonlocal diffusion and nonlocal mechanics.}
arXiv:1912.06668.



\bibitem{DiPaola} Di Paola, M.; Giuseppe F.; M. Zingales. { \it Physically-based approach to the mechanics of strong non-local linear elasticity theory.} Journal of Elasticity 97.2 (2009): 103--130.


\bibitem{Du} Du, Q.; Li, X.\,H.; Lu, J.; Tian, X. \emph{A quasi-nonlocal coupling method for nonlocal and local diffusion models}. SIAM J. Numer. Anal. 56 (2018), no.\,3, 1386--1404.


\bibitem{F} Fife, P. \emph{Some nonclassical trends in parabolic and parabolic-like evolutions}. In \lq\lq Trends in nonlinear analysis'', 153–191, Springer, Berlin, 2003.



\bibitem{Gal} Gal, C. G.; Warma, M. {\it Nonlocal transmission problems with fractional diffusion and boundary conditions on non-smooth interfaces.} Comm. Partial Differential Equations, 42(4), (2017), 579--625.


\bibitem{GQR} G\'{a}rriz, A.; Quir\'os, F.; Rossi, J. D. {\it Coupling local and nonlocal evolution equations.} Calc. Var. PDE, 
59(4), article 117, (2020), 1--25.

\bibitem{Han} Han, F.; Gilles L. { \it Coupling of nonlocal and local continuum models by the Arlequin approach.} Inter. Journal Numerical Meth. 
Engineering, 89(6), (2012): 671--685.

\bibitem{Hutson} Hutson, V.; Martinez, S.; Mischaikow, K.; Vickers, G. T. {\it The evolution of dispersal}. J. Math. Biology, 47(6), (2003),
483--517.



\bibitem{Kri} Kriventsov, D. {\it Regularity for a local-nonlocal transmission problem}. Arch. Ration.
Mech. Anal. 217 (2015), 1103--1195.

\bibitem{Mar} Marsden, J.E.; Hughes T.J.R. { 
Mathematical Foundations of Elasticity.}
Dover Publications, Rev. Ed. 2012.

\bibitem{MenDu} 
Mengesha, T.; Du, Q. {\it The bond-based peridynamic system with Dirichlet-type volume constraint}, Proc. Roy. Soc. Edinburgh Sect. A, 
(2014), Nro. 1, 161--186.


\bibitem{santos2020} dos Santos, B. C.; Oliva, S. M.; Rossi, J. D.
\newblock \emph{A local/nonlocal diffusion model}. To appear in Applicable Analysis.
\newblock arXiv preprint: 2003.02015, 2020.


\bibitem{Sel} Seleson, P.; Samir B.; Serge P.
{ \it A force-based coupling scheme for peridynamics and classical elasticity.} Computational Materials Science, 66, (2013), 34--49.

\bibitem{Sel2} Seleson, P.; Gunzburger, M. {\it Bridging methods for atomistic-to-continuum coupling and their implementation}.
Comm. Comput. Physics, 7(4), (2010), 831. 

\bibitem{Sel3} Seleson, P.; Gunzburger, M.; Parks, M.L. {\it Interface problems in nonlocal diffusion and sharp transitions
between local and nonlocal domains}. Comput. Methods Appl. Mech. Engineering,
266, (2013), 185--204.

\bibitem{Sil} Silling, S. A.; {\it Reformulation of elasticity theory for discontinuities and long-range forces}. Jour. Mech. Physics Solids, 48(1), 
 2000, 175---209,.

\bibitem{Sil1} Silling, S. A.; Lehoucq, R. B. { \it Peridynamic theory of solid mechanics}. In Advances in applied mechanics. Vol. 44, pp. 73--168. Elsevier, 2010.


\bibitem{Sil2} Silling, S. A.; Epton, M.; Weckner, O.; Xu, J.; 
Askari, E. {\it Peridynamic
states and constitutive modeling.} Jour. Elasticity, 
88, (2007), 151--184.

\bibitem{Strick} Strickland, C.; Gerhard D.; Patrick D. S.; { \it Modeling the presence probability of invasive plant species with nonlocal dispersal.} Jour.
Math. Biology, 69(2), (2014), 267--294.

\bibitem{TeMir} Temam, R. , Miranvielle, A. {Mathematical Modeling in Continuum Mechanics.} Cambridge Univ. Press, 2000.

\bibitem{W} Wang, X. \emph{Metastability and stability of patterns in a convolution model for phase transitions}.
J. Differential Equations, 183, (2002), no.\,2, 434--461.

\bibitem{Z} Zhang, L. \emph{Existence, uniqueness and exponential stability of traveling wave solutions of some integral differential equations
arising from neuronal networks}. J. Differential Equations, 197, (2004), no.\,1, 162--196.



\end{thebibliography}

\end{document}